\documentclass[preprint,12pt]{elsarticle}
\usepackage[english]{babel}

\usepackage{graphicx}
\usepackage{amsmath,amsfonts,amssymb,amsthm}
\usepackage{mathrsfs}

\newcommand{\ovln}[1]{\overline{#1}}

\newcommand{\angbr}[2]{\langle #1,#2 \rangle}

\newcommand{\freccia}[3]{#2\colon#1 \to #3}

\newcommand{\frecciainj}[3]{\xymatrix{#2 \colon #1  \ar@{^{(}->}[r] &  #3}}

\newcommand{\frecciasopra}[3]{#1\xrightarrow{#2} #3}

\newcommand{\pbmorph}[2]{#1^{\ast}#2} 

\newcommand{\duefreccia}[3]{\xymatrix@C=0.5cm{#2 \colon #1  \ar@{=>}[r] &  #3}}

\newcommand{\duemorfismo}[6]{\xymatrix@+1pc{
#1^{\op} \ar[rrd]^#2_{}="a" \ar[dd]_{#3^{\op}}\\
&& \infsl\\
#5^{\op}  \ar[rru]_#6^{}="b"
\ar_{#4}  "a";"b"}}
\newcommand{\comsquare}[8]{ \xymatrix@+1pc{ 
#1 \ar[r]^{#5} \ar[d]_{#6} & #2 \ar[d]^{#7} \\
#3 \ar[r]_{#8} & #4 
}}
\newcommand{\pullback}[8]{ \xymatrix@+1pc{ 
#1 \pullbackcorner \ar[r]^{#5} \ar[d]_{#6} & #2 \ar[d]^{#7} \\
#3 \ar[r]_{#8} & #4 
}}
\newcommand{\quadratocomm}[8]{ \xymatrix@+1pc{ 
#1 \ar[r]^{#5} \ar[d]_{#6} & #2 \ar[d]^{#7} \\
#3 \ar[r]_{#8} & #4 
}}
\newcommand{\comsquarelargo}[8]{ \xymatrix@+1pc{ 
#1 \ar[rr]^{#5} \ar[d]_{#6} && #2 \ar[d]^{#7} \\
#3 \ar[rr]_{#8} && #4 
}}
\newcommand{\parallelmorphisms}[4]{\xymatrix@+1pc{
#1 \ar @<+4pt>[r]^{#2} \ar @<-4pt>[r]_{#3} & #4
}}
\newcommand{\relation}[4]{\xymatrix@+1pc{
\angbr{#2}{#3}\colon #1 \ar @<+4pt>[r] \ar @<-4pt>[r] & #4
}}
\newcommand{\frecceparalleleopposte}[4]{\xymatrix@+1pc{
#1 \ar@<+4pt>[r]^{#2} \ar@<-4pt>@{<-}[r]_{#3} & #4
}}
\newcommand{\equalizer}[6]{\xymatrix@+1pc{
#1 \ar[r]^{#2} & #3 \ar @<+4pt>[r]^{#4} \ar @<-4pt>[r]_{#5} & #6
}}
\newcommand{\coequalizer}[6]{\xymatrix@+1pc{
 #1 \ar @<+4pt>[r]^{#2} \ar @<-4pt>[r]_{#3} & #4 \ar[r]^{#5} & #6
}}

\newcommand{\sottoggetto}[2]{\xymatrix{
#1 \ar@{>->}[r] & #2
}}

\newcommand{\pullbackcorner}[1][ul]{\save*!/#1+1.2pc/#1:(1,-1)@^{|-}\restore}

\def\pr{\pi}
\def\id{\operatorname{ id}}
\def\op{\operatorname{ op}}

\def\mC{\mathcal{C}}

\def\mD{\mathcal{D}}
\def\mE{\mathcal{E}}
\def\mG{\mathcal{G}}

\def\sT{\mathsf{T}}

\def\Sub{\mathsf{Sub}}

\def\hey{\mathsf{Hey}}
\def\infsl{\mathsf{InfSl}}
\newcommand{\pca}[1]{\mathbb{#1}}

\def\set{\mathsf{Set}}
      
\newcommand{\sh}[1]{\mathsf{Sh}(#1)}
\newcommand{\sheaves}[1]{\mathsf{Sh}(#1)}

\newcommand{\shtopology}[2]{\mathsf{Sh}_{#2}(#1)}
\newcommand{\presh}[1]{\mathsf{E}_{#1}}
\newcommand{\assembly}[1]{\mathpzc{Asm}(#1)}
\newcommand{\parassembly}[1]{\mathpzc{PAsm}(#1)}
\newcommand{\EED}{\mathsf{ExDoc}}

\newcommand{\Excat}{\mathsf{ExCat}}
\newcommand{\regdoctrine}[1]{\mathsf{Reg}(#1)}
\newcommand{\Reg}{\mathsf{RegCat}}

\newcommand{\doctrine}[2]{#2\colon #1^{\op}\longrightarrow\infsl}
\newcommand{\hyperdoctrine}[2]{#2\colon #1^{\op}\longrightarrow\hey}

\newcommand{\exlex}[1]{(#1)_{\mathsf{ex}/\mathsf{lex}}}
\newcommand{\exreg}[1]{(#1)_{\mathsf{ex}/\mathsf{reg}}}
\newcommand{\reglex}[1]{(#1)_{\mathsf{reg}/\mathsf{lex}}}
\newcommand{\loc}{\mathsf{L}}

\newcommand{\supercomp}{\mathsf{S}}
\newcommand{\compexfull}[1]{{#1}^{\exists}}

\newcommand{\powerset}{\mathscr{P}}

\def\fullprDoc{\mathsf{LexDoc}}

\usepackage{todonotes}

\DeclareMathAlphabet{\mathpzc}{OT1}{pzc}{m}{it}
\usepackage{enumitem} 
\usepackage{quiver}
\usepackage{tikz-cd}
\usepackage{aliascnt}
\usepackage{natbib}
\usepackage[letterpaper,top=2cm,bottom=2cm,left=3cm,right=3cm,marginparwidth=1.75cm]{geometry}
\usepackage[unicode]{hyperref}
\usepackage{xcolor}
\usepackage[nameinlink]{cleveref}

\newtheorem{theorem}{Theorem}[section]
\newaliascnt{cor}{theorem}
\newtheorem{cor}[cor]{Corollary}
\aliascntresetthe{cor}
\newaliascnt{lemma}{theorem}
\newtheorem{lemma}[lemma]{Lemma}
\aliascntresetthe{lemma}
\newaliascnt{proposition}{theorem}
\newtheorem{proposition}[proposition]{Proposition}
\aliascntresetthe{proposition}
\newaliascnt{fact}{theorem}

\aliascntresetthe{fact}

\theoremstyle{definition} 
\newaliascnt{definition}{theorem}
\newtheorem{definition}[definition]{Definition}
\aliascntresetthe{definition}
\newaliascnt{axiom}{theorem}
\newtheorem{axiom}[axiom]{Axiom}
\aliascntresetthe{axiom}
\newaliascnt{Rule}{theorem}
\newtheorem{Rule}[Rule]{Rule}
\aliascntresetthe{Rule}
\newaliascnt{remark}{theorem}
\newtheorem{remark}[remark]{Remark}
\aliascntresetthe{remark}
\newaliascnt{example}{theorem}
\newtheorem{example}[example]{Example}
\aliascntresetthe{example}
\newaliascnt{exercise}{theorem}

\aliascntresetthe{exercise}
\newaliascnt{conjecture}{theorem}

\aliascntresetthe{conjecture}

\crefname{theorem}{theorem}{theorems}
\Crefname{theorem}{Theorem}{Theorems}
\crefname{cor}{corollary}{corollaries}
\Crefname{cor}{Corollary}{Corollaries}
\crefname{lemma}{lemma}{lemmas}
\Crefname{lemma}{Lemma}{Lemmas}
\crefname{proposition}{proposition}{propositions}
\Crefname{proposition}{Proposition}{Propositions}
\crefname{fact}{fact}{facts}
\Crefname{fact}{Fact}{Facts}
\crefname{definition}{definition}{definitions}
\Crefname{definition}{Definition}{Definitions}
\crefname{axiom}{axiom}{axioms}
\Crefname{axiom}{Axiom}{Axioms}
\crefname{Rule}{rule}{rules}
\Crefname{Rule}{Rule}{Rules}
\crefname{remark}{remark}{remarks}
\Crefname{remark}{Remark}{Remarks}
\crefname{example}{example}{examples}
\Crefname{example}{Example}{Examples}
\crefname{exercise}{exercise}{exercises}
\Crefname{exercise}{Exercise}{Exercises} 
\crefname{conjecture}{conjecture}{conjectures}
\Crefname{conjecture}{Conjecture}{Conjectures}

\begin{document}

\begin{frontmatter}
\title{An Algebraic Abstraction of the Localic Sheafification \\via the Tripos-to-Topos Construction}
\author[1]{M.E. Maietti}
\author[1]{D. Trotta}
\address[1]{University of Padova, Department of Mathematics}

\begin{abstract}
  Localic and realizability toposes are two central classes of toposes in categorical logic, both arising from the Hyland, Johnstone and Pitts tripos-to-topos construction. We investigate the common geometric structure underlying these classes by developing an algebraic abstraction of localic presheaves, localic sheafification, and the role of supercompactification in the Comparison Lemma. The resulting framework applies to a broad class of toposes arising from the tripos-to-topos construction, including all $\set$-based triposes. These results establish a common geometric framework encompassing both localic and realizability toposes.

\end{abstract}
\end{frontmatter}

\tableofcontents
\section{Introduction}
Localic and realizability toposes are two fundamental classes of toposes in categorical logic. The main distinction between them is that localic toposes are Grothendieck toposes of sheaves, whereas realizability toposes are generally not Grothendieck toposes.

In the 1980s, Hyland, Johnstone, and Pitts introduced the notion of \emph{tripos}~\cite{TT} to explain, from an abstract perspective, how Higgs’ description of localic sheaf toposes as $\mathsf{H}$-valued sets~\cite{higgs84} and Hyland’s Effective topos $\mathsf{Eff}$~\cite{HYLAND1982165} are both instances of the same general construction. 

Inspired by the works of Higgs~\cite{higgs84} and Fourman and Scott~\cite{Fourman79}, they defined a particular family of Lawvere’s hyperdoctrines~\cite{AF,EHCSAF,DACCC}, called \emph{triposes}, together with the \emph{tripos-to-topos} construction, which associates a topos $\sT_{P}$ with every tripos $\hyperdoctrine{\mC}{P}$.  Both localic and realizability toposes arise as instances of this general construction for suitable triposes.

After their introduction in the two seminal works~\cite{TT,TTT}, the notion of tripos as well as the tripos-to-topos construction have been widely studied and used in categorical logic. They have found important applications in realizability theory \cite{van_Oosten_realizability,Hofstra2006,Frey2014AFS,miquel_2020}, and  the universal property of the tripos-to-topos construction has been analysed in \cite{FREY2015232}. 

Moreover, in a series of works \cite{EQC,UEC,TECH}, the first author and Rosolini observed that the tripos-to-topos construction can be generalized to \emph{elementary and existential doctrines} and, in such a context, it has the universal property of being the exact completion of this class of doctrines.

A key insight that seems to disappear in this picture is the natural geometric description of the notion of \emph{localic sheaf}. Indeed, localic sheaves are usually presented as presheaves satisfying a gluing condition, whereas the tripos-to-topos construction provides an explicit procedure that, given a tripos, produces a topos. 

The main purpose of this work is to recover the geometric content of toposes obtained from the tripos-to-topos construction and  to further investigate the common geometric structures underlying these two classes of toposes.

\paragraph{Localic Landscape} Given a locale $\loc$, two well-known categories naturally arise: the category of \emph{sheaves} $\sheaves{\loc}$ and the category of \emph{presheaves} $\mathsf{PSh}(\loc)$. They are related by the \emph{sheafification} adjunction
 \[\begin{tikzcd}
					   \mathsf{PSh}(\loc) && \sheaves{\loc}
					\arrow[""{name=0, anchor=center, inner sep=0}, "",curve={height=-14pt}, from=1-1, to=1-3]
					\arrow[""{name=1, anchor=center, inner sep=0}, "",curve={height=-14pt}, hook',from=1-3, to=1-1]
					\arrow["\dashv"{anchor=center, rotate=-90}, draw=none, from=1, to=0]
				\end{tikzcd}\]
Moreover, the classical Comparison Lemma yields an equivalence
\[ \mathsf{PSh}(\loc) \equiv \sh{D(\loc)}\]
 where $D(\loc)$ is the \emph{supercompact (or supercoherent) locale} obtained by applying the construction $D(-)$,  which freely adjoins arbitrary joins to an inf-semilattice (i.e. a poset with finite infs), to $\loc$. See, for example, \cite{BANASCHEWSKI199145} and \cite[Sec.C.1.1.]{SAE2} for the details of this construction. 
The interaction between these constructions can be summarized by the following diagram, which commutes up to equivalence:
\[\begin{tikzcd}
	\mathsf{Loc} && \mathsf{Topoi} \\
	& \mathsf{Loc}
	\arrow["\mathsf{PSh}(-)",squiggly, from=1-1, to=1-3]
	\arrow["D(-)"', squiggly, from=1-1, to=2-2]
	\arrow["\sheaves{-}"', squiggly, from=2-2, to=1-3]
\end{tikzcd}\]

\paragraph{Notation} Throughout the paper we use squiggly arrows rather than straight arrows to emphasize that we are referring only to constructions on objects, and not to genuine 2-functors or pseudofunctors. Accordingly, we say that a diagram involving squiggly arrows commutes whenever it commutes up to equivalence. These diagrams are intended as a guide to help the reader visualize the interaction between the various constructions considered throughout the paper.

\paragraph{Our main purpose} Our goal is to extend the entire picture above to the setting of triposes. More precisely, we identify an abstract notion of localic presheaf topos, introduce an analogue of localic supercompactification, establish an abstract version of the Comparison Lemma, and finally obtain a corresponding notion of sheafification by combining these ingredients.



To achieve this goal, we proceed in two stages. We first develop the required constructions in the more general setting of doctrines. We then show how these constructions restrict to triposes, where they yield the desired generalization of the localic picture.

\paragraph{Abstract localic presheaf category (\Cref{Sec:ab loc presh})} 
Our starting point is the observation that the localic presheaf category admits a purely categorical characterization. Indeed, for every locale $\loc$, the localic presheaf topos is equivalent to the $\mathsf{ex}/\mathsf{lex}$-completion, in the sense of Carboni~\cite{REC,SFEC}, of the category of elements $\mG_{\loc^{(-)}}$ associated with the localic tripos $\hyperdoctrine{\set}{\loc^{(-)}}$; that is,
\[
\mathsf{PSh}(\loc)\equiv \exlex{\mG_{\loc^{(-)}}}.
\]

Motivated by this characterization, we define the \emph{exact category of elements} $\presh{P}$ of a lex primary doctrine $P$, that is, a functor $\doctrine{\mC}{P}$ from a lex category to the category $\infsl$ of inf-semilattices, as the $\mathsf{ex}/\mathsf{lex}$-completion
\[
\presh{P}:=\exlex{\mG_P},
\]
of the category of elements $\mG_P$ associated with the doctrine $P$ (\Cref{def:P-presheaves}).

\paragraph{Abstract localic supercompactification (\Cref{Sec: ab localic supercomp})} 
Our second key observation is that the full existential completion plays, for lex primary doctrines, the same role that the supercompactification construction plays for locales. More precisely, for the lex primary doctrine $\doctrine{\set}{\mathsf{S}^{(-)}}$ associated with an inf-semilattice $\mathsf{S}$, its full existential completion coincides with the localic doctrine associated with the supercompactification $D(\mathsf{S})$, as proved in~\cite[Thm.~7.32]{MaiettiTrotta21} and recalled in \Cref{ex:supercoherent localic doctrine is full ex comp}.

We recall that the full existential completion freely adds
existential quantifiers along all morphisms of the base to a lex
primary doctrine. It is an instance of the more general construction
for doctrines studied in~\cite{ECRT}, which freely adds left adjoints
along a class of morphisms containing all isomorphisms and closed
under composition and pullback.
From a fibrational perspective, the full existential completion
can be obtained as the fibrewise poset reflection of the classical
Family construction for fibrations.
For further details on the latter construction and its applications,
we refer, for example, to~\cite{StreetFamily,hofstra2011} and
to~\cite[Sec.~B.1.4, Prop.~1.4.16]{SAE}.

\paragraph{An Abstract Comparison Lemma (\Cref{sec:com lemma})}
Our first main theorem is an abstract analogue of the classical Comparison Lemma. More precisely, we prove that
\[
\presh{P}\equiv \sT_{\compexfull{P}},
\]
where $\compexfull{P}$ denotes the full existential completion of $P$ and $\sT_{\compexfull{P}}$ its tripos-to-topos construction (\Cref{thm:characterisation tripostotopos full ex comp}). 

This result can be summarized by the following commutative diagram:
\[
\begin{tikzcd}
	\fullprDoc && \Excat \\
	& \EED
	\arrow["\presh{(-)}",squiggly, from=1-1, to=1-3]
	\arrow["\compexfull{(-)}"',squiggly, from=1-1, to=2-2]
	\arrow["\sT_{(-)}"', squiggly,from=2-2, to=1-3]
\end{tikzcd}
\]
where $\fullprDoc$ denotes the 2-category of lex primary doctrines, $\EED$ the 2-category of full existential doctrines, and $\Excat$ the 2-category of exact categories.

\paragraph{Notation} Following the terminology introduced in \cite{MaiettiTrotta21}, we use the adjective “full” before “existential doctrine” to emphasize that reindexing along every morphism has a left adjoint, and that these adjoints satisfy both the Beck-Chevalley condition for arbitrary pullback squares and Frobenius reciprocity. These requirements are stronger than those imposed on elementary existential doctrines in the sense of \cite{QCFF,EQC}: in general, the latter need not satisfy the Beck-Chevalley condition for arbitrary pullback squares, as shown in \cite[Rem. 6.4]{MaiettiTrotta21}. The Beck-Chevalley condition for arbitrary pullback squares, together with Frobenius reciprocity, will play a crucial role in the proofs of the main results of this paper. Similarly, we use the adjective ``full'' for first-order
hyperdoctrines and triposes to emphasize that the Beck-Chevalley
condition is required to hold for arbitrary pullback squares.

\paragraph{Moving to triposes and toposes}
The previous analysis provides a clear picture for lex primary doctrines and exact categories. However, the constructions $\presh{(-)}$ and $\compexfull{(-)}$ do not, in general, produce a topos and a tripos, respectively. The goal of the second part of the paper is to characterize and study those doctrines and triposes for which $\presh{(-)}$ yields a topos and $\compexfull{(-)}$ yields a tripos.

\paragraph{Characterizing when $\presh{P}$ is a topos and $\compexfull{P}$ is a tripos (\Cref{Sec char abstr localic presh tops})} 
In the localic setting, the supercompactification of an inf-semilattice is always a locale, i.e. for every inf-semilattice $\mathsf{S}$, $D(\mathsf{S})$ is a locale. This property, however, does not extend to arbitrary lex primary doctrines: if $P$ is a lex primary doctrine, its full existential completion $\compexfull{P}$ is not, in general, a tripos.

We identify the precise condition under which the full existential completion of a lex primary doctrine is a tripos. We call a doctrine \emph{$\exists$-supercompactifiable} if its category of elements $\mG_P$ has weak dependent products and a generic proof in the sense of~\cite{MENNI2002187,MENNI2007511}, (\Cref{def:exists-sheaf tripos}). Our main result shows that the following conditions are equivalent: $P$ is $\exists$-supercompactifiable, $\compexfull{P}$ is a tripos, and $\presh{P}$ is a topos (\Cref{thm:comp ex tripos is equivalent to}).

These results are summarized by the following commutative diagram:

\[\begin{tikzcd}
	\exists\mbox{-}\fullprDoc && \mathsf{Topoi} \\
	& \mathsf{Trip}
	\arrow["\presh{(-)}",squiggly, from=1-1, to=1-3]
	\arrow["\compexfull{(-)}"',squiggly, from=1-1, to=2-2]
	\arrow["\sT_{(-)}"',squiggly, from=2-2, to=1-3]
\end{tikzcd}\]
where $\exists\mbox{-}\fullprDoc$ denotes the 2-full subcategory of $\fullprDoc$ consisting of $\exists$-supercompactifiable doctrines, $\mathsf{Trip}$ denotes the 2-full subcategory of $\EED$ consisting of triposes.

Restricting this diagram to the 2-full subcategory $\exists\mbox{-}\mathsf{Trip}$ of $\mathsf{Trip}$ consisting of $\exists$-supercompactifiable triposes yields the following diagram:

\[\begin{tikzcd}
	\exists\mbox{-} \mathsf{Trip} && \mathsf{Topoi} \\
	& \mathsf{Trip}
	\arrow["\presh{(-)}",squiggly, from=1-1, to=1-3]
	\arrow["\compexfull{(-)}"',squiggly, from=1-1, to=2-2]
	\arrow["\sT_{(-)}"',squiggly, from=2-2, to=1-3]
\end{tikzcd}\]

\paragraph{A characterization of $\exists$-supercompactifiable triposes (\Cref{Sec char ex supercomp triposes})}
The supercompactification of a locale is again a locale, so the construction can be iterated indefinitely. By contrast, our previous characterization does not yet imply that the full existential completion of an $\exists$-supercompactifiable doctrine is again $\exists$-supercompactifiable; it only guarantees that it is a tripos.

Moreover, in many concrete examples, verifying that a tripos is $\exists$-supercompactifiable can be rather involved, because of the interaction of the structure of the base category with the elements of the fibres in the category of elements $\mG_P$.

To overcome these difficulties, we establish a new characterization of $\exists$-supercompactifiable triposes, which resolves both of these issues.

More precisely, we prove that a tripos is $\exists$-supercompactifiable if and only if its base category has weak dependent products and a generic proof (\Cref{thm:main}). Since the full existential completion leaves the base category unchanged, this characterization immediately implies that the full existential completion of an $\exists$-supercompactifiable doctrine is again $\exists$-supercompactifiable (\Cref{cor: stability of ex comp for supercompact doc}).
These results are summarized by the following commutative diagram:

\[\begin{tikzcd}
	\exists\mbox{-}\fullprDoc && \mathsf{Topoi} \\
	& \exists\mbox{-}\mathsf{Trip}
	\arrow["\presh{(-)}",squiggly, from=1-1, to=1-3]
	\arrow["\compexfull{(-)}"',squiggly, from=1-1, to=2-2]
	\arrow["\sT_{(-)}"',squiggly, from=2-2, to=1-3]
\end{tikzcd}\]
Restricting the previous diagram yields the following commutative diagram:\[\begin{tikzcd}
	\exists\mbox{-}\mathsf{Trip} && \mathsf{Topoi} \\
	& \exists\mbox{-}\mathsf{Trip}
	\arrow["\presh{(-)}",squiggly, from=1-1, to=1-3]
	\arrow["\compexfull{(-)}"',squiggly, from=1-1, to=2-2]
	\arrow["\sT_{(-)}"',squiggly, from=2-2, to=1-3]
\end{tikzcd}\]
Our characterization also shows that the class of $\exists$-supercompactifiable triposes is remarkably broad. Assuming the axiom of choice, it includes all triposes over $\set$ (\Cref{cor_every_set_base_tripos_is_sheaf_tripos}), and hence all realizability triposes, besides localic ones.  Prominent examples include: the Modified Realizability tripos \cite{Hylandmodifiedreal,VANOOSTENmodifiedreal}, the Extensional Realizability tripos \cite{VANOOSTENereal}, the dialectica tripos \cite{Biering2008}, the Krivine tripos \cite{STREICHER_2013}. A notable example over a base other than $\set$ is the tripos of extended Weihrauch degrees recently introduced in~\cite{maschitrottao2025}.

Since localic triposes are $\exists$-supercompactifiable, the previous diagram genuinely extends the localic picture and fits into the following commutative diagram summarizing our results.

\[\begin{tikzcd}
	&  \mathsf{Topoi} &&& \mathsf{Topoi}  \\
	\exists\mbox{-}\mathsf{Trip}  &&& \mathsf{Loc} \\
	& \exists\mbox{-}\mathsf{Trip} &&& \mathsf{Loc}
	\arrow[squiggly, hook', from=1-5, to=1-2]
	\arrow["\presh{(-)}",squiggly, from=2-1, to=1-2]
	\arrow["\compexfull{(-)}"',squiggly, from=2-1, to=3-2]
	\arrow["\mathsf{PSh}(-)",squiggly, from=2-4, to=1-5]
	\arrow[squiggly, hook', from=2-4, to=2-1]
	\arrow["D(-)"', squiggly, from=2-4, to=3-5]
	\arrow["\sT_{(-)}"'{pos=0.6},squiggly, from=3-2, to=1-2]
	\arrow["\mathsf{Sh}(-)"'{pos=0.6},squiggly, from=3-5, to=1-5]
	\arrow[squiggly, hook', from=3-5, to=3-2]
\end{tikzcd}\]

where the assignment from $\mathsf{Loc}$ into $\exists\mbox{-}\mathsf{Trip}$ is given by $\loc \mapsto \loc^{(-)}$ and the assignment of $\mathsf{Topoi}$ into itself is the identity.

Finally, we further demonstrate the breadth of the class of $\exists$-supercompactifiable triposes by proving an analogue of the fundamental theorem of toposes: these triposes are closed under localization (\Cref{cor localization}).

\paragraph{Abstracting Localic Sheafification (\Cref{sec abstract sheafification})}
We conclude by showing that the toposes arising from $\exists$-supercompactifiable triposes admit a presentation as toposes of $j$-sheaves for a suitable Lawvere-Tierney topology, which specializes in the localic case to the Lawvere-Tierney topology inducing the classical localic sheafification. This result further reveals the geometric nature of the toposes arising from the tripos-to-topos construction.

More precisely, we prove that, whenever a tripos $P$ is $\exists$-supercompactifiable, there is a geometric embedding
\[
P\hookrightarrow \compexfull{P},
\]
which induces a geometric embedding of toposes
\[
\sT_P\hookrightarrow \sT_{\compexfull{P}}.
\]
This is proved in \Cref{thm:main 0}.

As a consequence, for every $\exists$-supercompactifiable tripos $P$, there exists a Lawvere-Tierney topology $j$ on $\sT_{\compexfull{P}}$ such that
\[
\sT_P \equiv \shtopology{\sT_{\compexfull{P}}}{j}
\]
This is presented in \Cref{cor_sheafification_Lw_Tr}. When specialized to localic triposes, this result recovers exactly the classical presentation of sheaves on a locale as presheaves satisfying the usual sheaf gluing condition. 

Thus, our result extends the classical localic sheafification, together with its geometric interpretation of localic sheaves as presheaves satisfying a gluing condition, to the setting of toposes arising from $\exists$-supercompactifiable triposes.

Since every realizability tripos is $\exists$-supercompactifiable, our result uncovers a common geometric structure underlying realizability and localic toposes.

\paragraph{Outline of the paper}
\Cref{sec:prelim on doc} recalls the necessary background on doctrines, triposes, and the tripos-to-topos construction. \Cref{Sec:ab loc presh} and~\Cref{Sec: ab localic supercomp} develop the doctrinal counterparts of the category of localic presheaves and of localic supercompactification, respectively. \Cref{sec:com lemma} establishes an abstract version of the Comparison Lemma. \Cref{Sec char abstr localic presh tops} characterizes the doctrines for which the full existential completion is a tripos and the associated abstract localic presheaf category is a topos, while \Cref{Sec char ex supercomp triposes} specializes the theory to triposes, obtaining a characterization of $\exists$-supercompactifiable triposes. Finally, \Cref{sec abstract sheafification} establishes the abstract counterpart of localic sheafification, recovering the classical localic picture as a special case. \Cref{conclusioni} concludes the paper with some directions for future work.

\section{Preliminaries on doctrines and triposes}\label{sec:prelim on doc}
In this section, we briefly provide some categorical background for analyzing the tripos-to-topos construction via completions of doctrines and their relationship to geometric morphisms between toposes. We begin by recalling the completions of doctrines that we use to prove the main theorems of the paper.

\subsection{Lex primary and full existential doctrines}
The notion of a hyperdoctrine was introduced by F.~W.~Lawvere in a series of seminal papers~\cite{AF,EHCSAF}.
We recall here some definitions that will be useful in what follows.
Further details on the theory of elementary and existential doctrines can be found in~\cite{QCFF,EQC,UEC,TECH,EMMENEGGER2020}.



\begin{definition}\label{def primary doctrine}\label{def lex primary doctrine}
A \textbf{lex primary doctrine}~is a functor
$\doctrine{\mC}{P}$ from the opposite of a category $\mC$ with finite limits to the category $\mathsf{InfSl}$ of inf-semilattices, namely the category whose objects are posets with finite infs, and whose morphisms are monotone functions preserving finite infs.
\end{definition}

\begin{definition}\label{def:morphism of primary doctrine}
Let $\doctrine{\mC}{P}$ and $\doctrine{\mD}{R}$ be two lex primary doctrines. A \textbf{lex primary morphism} of doctrines is given by a pair $(F,\mathfrak{b})$ 

\[\begin{tikzcd}
   \mC^{\op} \\
   && \infsl \\
   \mD^{\op}
   \arrow[""{name=0, anchor=center, inner sep=0}, "R"', from=3-1, to=2-3]
   \arrow[""{name=1, anchor=center, inner sep=0}, "P", from=1-1, to=2-3]
   \arrow["F^{\op}"', from=1-1, to=3-1]
   \arrow["\mathfrak{b}"', shorten <=4pt, shorten >=4pt, from=1, to=0]
\end{tikzcd}\]
where
\begin{itemize}
\item $\freccia{\mC}{F}{\mD}$ is a functor preserving finite limits;
\item $\freccia{ P}{\mathfrak{b}}{R\circ F^{\op}}$ is a natural transformation.
\end{itemize}

\end{definition}
\begin{definition}\label{def:doctrine transformation}
Let us consider two lex primary doctrines $\doctrine{\mC}{P}$ and $\doctrine{\mD}{R}$, and two lex primary morphisms $\freccia{P}{(F,\mathfrak{b}),(G,\mathfrak{c})}{R}$. A \textbf{doctrine transformation} is a natural transformation $\freccia{F}{\theta}{G}$ such that
\[ \mathfrak{b}_A(\alpha)\leq R_{\theta_A}(\mathfrak{c}_A(\alpha))\]
for every $\alpha$ of $P(A)$.
\end{definition}
We denote by $\fullprDoc$ the 2-category of lex primary doctrines, lex primary morphisms and doctrine transformations.

The following example of primary doctrine is introduced in \cite{Hofstra2006}.
\begin{example}\label{ex: primary doctrine pca}
Let $\pca{A}$ be a partial combinatory algebra (pca). We can define a lex primary doctrine $\doctrine{\set}{\pca{A}^{(-)}}$ assigning to a set $X$ the set $\pca{A}^X$ of functions from $X$ to $\pca{A}$. Given two elements $\alpha,\beta\in \pca{A}^X$,  we have that $\alpha\leq \beta$ if there exists an element $a\in \pca{A}$ such that for every $x\in X$ we have that $a\cdot \alpha(x)$ is defined and $a\cdot \alpha(x)=\beta(x)$. This defines a preorder on $\pca{A}^X$, and hence we obtain the desired doctrine by considering its poset reflection.
\end{example}
\begin{example}\label{ex:lex primary doctrine localic}
    Let $\supercomp$ be an inf-semilattice. We can define a lex primary doctrine $\doctrine{\set}{\supercomp^{(-)}}$ assigning to a set $X$ the set $\supercomp^X$ of functions from $X$ to $\supercomp$. Given two elements $f,g\in \supercomp^X$,  we have that $f\leq g$ if $f(x)\leq g(x)$ for every $x\in X$.
\end{example}
We recall the notion of \emph{full existential doctrine} from  \cite{MaiettiTrotta21}. Following the terminology introduced there, we use the adjective “full” to emphasize that reindexing along every morphism has a left adjoint, and that these adjoints satisfy both the Beck-Chevalley condition for arbitrary pullback squares and Frobenius reciprocity. These requirements are stronger than those imposed on elementary existential doctrines in the sense of \cite{QCFF,EQC}: in general, the latter need not satisfy the Beck-Chevalley condition for arbitrary pullback squares, as shown in \cite[Rem. 6.4]{MaiettiTrotta21}.

\begin{definition}\label{def existential doctrine}
A lex primary doctrine $\doctrine{\mC}{P}$ is \textbf{full existential} if, for every arrow $\freccia{A}{f}{B}$ in $\mC$, the functor
$ \freccia{P(B)}{{P_{f}}}{P(A)}$
has a left adjoint $\exists_{f}$, and these satisfy \emph{Beck-Chevalley condition} and \emph{Frobenius reciprocity}.

\end{definition}

\begin{remark}
    Notice that in a full existential doctrine $\doctrine{\mC}{P}$ we can define an equality predicate $\delta_X:=\exists_{\angbr{\id_X}{\id_X}}(\top_X)$ in $ P(X\times X)$ for every object $X$ of $\mC$.
    
    \end{remark}

\begin{definition}
A lex primary morphism is called \textbf{full existential} whenever the natural transformation $\freccia{P}{\mathfrak{b}}{R\circ F^{\op}}$ commutes with left adjoints along every arrow.

\end{definition}
We denote by $\EED$ the  2-category of full existential doctrines, full existential morphisms and doctrine transformations.

The following examples are discussed in \cite{AF,UEC}.
\begin{example}[subobjects doctrine]\label{ex:sub doctrine}
 Let $\mC$ be a category with finite limits. The functor $\doctrine{\mC}{{\Sub_{\mC}}}$
assigns to an object $A$ in $\mC$ the poset $\Sub_{\mC}(A)$ of subobjects of $A$ in $\mC$ and, for an arrow $\freccia{B}{f}{A}$ the morphism $\freccia{\Sub_{\mC}(A)}{\Sub_{\mC}(f)}{\Sub_{\mC}(B)}$ is given by pulling a subobject back along $f$.  This is a full existential elementary doctrine if and only if the category $\mC$ is regular. 
\end{example}
\begin{example}[weak-subobjects doctrine]\label{ex:weak sub doctrine}
Let $\mC$ be a category with finite limits. The functor $\doctrine{\mC}{{\Psi_{\mC}}}$ assigns to an object $A$ in $\mC$ the poset reflection of the slice category $\mC/A$, and for an arrow $\freccia{B}{f}{A}$, the homomorphism $\freccia{\Psi_{\mC}(A)}{\Psi_{\mC}(f)}{\Psi_{\mC}(B)}$ is defined via the pullback. This doctrine is full existential, and the left adjoints are given by postcomposition.
\end{example}
The following examples are discussed in  \cite{TT,TTT}.
\begin{example}[localic doctrine]\label{ex:localic tripos}
Let $\loc$ be a locale. The lex primary doctrine $\doctrine{\set}{\loc^{(-)}}$ defined in \Cref{ex:lex primary doctrine localic} is a full existential doctrine.
\end{example}

\begin{example}[realizability doctrine]\label{ex:realizability tripos}
Let $\pca{A}$ be a pca, we can consider the realizability doctrine $\doctrine{\set}{\mathcal{P}}$: for each set $X$, $(\powerset(\pca{A})^X,\leq)$ is defined as the set of functions from $X$ to the powerset $\powerset(\pca{A})$ of $\pca{A}$ and, given two elements $\alpha$ and $\beta$ of $\powerset(\pca{A})^X$, we say that $\alpha\leq \beta$ if there exists an element $\ovln{a}\in \pca{A}$ such that for all $x\in X$ and all $a\in \alpha (x)$, $\ovln{a}\cdot a$ is defined and it is an element of $\beta (x)$. As in \Cref{ex: primary doctrine pca}, this defines a preorder, so we have to consider the corresponding poset reflection.
\end{example}

\subsection{Hyperdoctrines and triposes}

The notion of \emph{tripos} and the \emph{tripos-to-topos construction} were originally introduced in \cite{TT} and revisited in \cite{TTT} in order to generalize the construction of the category of sheaves of a locale. Over the past few years, there has been an increasing focus on its universal properties in
\cite{UEC,TECH,FREY2015232}, which we will exploit in our work.
In particular, we will make use of the fact, shown in \cite{UEC}, that the tripos-to-topos
construction coincides with the {\it exact completion} of a tripos viewed as {\it an elementary existential doctrine}.

Furthermore,  in \cite{TT} they also introduced a notion of geometric morphism between $\set$-based triposes, which
induces  a geometric morphism between their generated  toposes.  This also specializes to the notion of geometric embedding from a tripos $P$ to a tripos $R$, which  extends to a geometric embedding between the corresponding toposes, to view the topos
generated by $P$ as a topos of internal sheaves  on the topos generated by $R$ for the Lawvere-Tierney topology  induced by the geometric embedding, see e.g.  \cite{SGL}.
We start by recalling the main concepts regarding first-order hyperdoctrines and triposes, mainly following the notation used in \cite{TTT}.
\begin{definition}\label{def:hyperdoctrine}
A \textbf{full first-order hyperdoctrine} is a full existential doctrine $\doctrine{\mC}{P}$  such that
\begin{itemize}
\item for every object $A$ of $\mC$ the fibre $P(A)$ is a Heyting algebra, and for every arrow $\freccia{A}{f}{B}$ of $\mC$, $\freccia{P(B)}{P_f}{P(A)}$ is a morphism of Heyting algebras;
\item for any arrow $\freccia{ A}{f}{B}$, the functor
\[ \freccia{P(B)}{{P_{f}}}{P(A)}\]
has both a left adjoint $\exists_{f}$ and a right adjoint $\forall_{f}$. Moreover, these adjoints have to satisfy \emph{Beck-Chevalley condition}.
\end{itemize}

\end{definition}

\begin{definition}
A lex primary doctrine $\doctrine{\mC}{P}$ has a
\textbf{weak predicate classifier} if there exists an object $\Omega$ of $\mC$ together with an element $\in$ of $P(\Omega)$ such that for every object $A$ of $\mC$ and every $\alpha$ of $P(A)$ there exists a morphism $\freccia{A}{\{\alpha\}}{\Omega}$ such that $P_{ \{\alpha\}}(\in)=\alpha$.
\end{definition}
\begin{definition}
A lex primary doctrine $\doctrine{\mC}{P}$ has
\textbf{weak power objects} if for every object $X$ of $\mC$ there exists an object $\mathrm{P}X$ and an element $\in_X$ of $P(X\times\mathrm{P}X)$ such that for every $\alpha$ of $P(X\times Y)$ there exists an arrow $\freccia{Y}{\{\alpha\}_X}{\mathrm{P}X}$ such that $\alpha=P_{\id_X\times \{\alpha\}_X}(\in_X)$.

\end{definition}
Notice that if $\doctrine{\mC}{P}$ has weak power objects, then it has a weak predicate classifier given by $\mathrm{P}1$ and $\in_1$.
The converse holds when the base category is cartesian closed \cite{TTT}, and this extends also to the weak case:
\begin{lemma}\label{rem:weak cart closed power ob iff pred clas}
If $\doctrine{\mC}{P}$ has a weak predicate classifier and the base category is weakly cartesian closed then, for every object $A$ of $\mC$, we can define an object $\mathrm{P}A:=\Omega^A$ and an element $\in_A:=P_{\mathsf{ev}}(\in)$ of $P(A\times \mathrm{P}A)$
where $\freccia{A\times \Omega^A}{\mathsf{ev}}{\Omega}$ is the evaluation arrow, and these assignments give to $P$ the structure of weak power objects.
\end{lemma}
\begin{definition}\label{def:tripos}
A full first-order hyperdoctrine $\doctrine{\mC}{P}$ is called a \textbf{full tripos} if it has  weak power objects. 
\end{definition}

\begin{example}
The localic doctrine $\doctrine{\set}{\loc^{(-)}}$ defined in \Cref{ex:localic tripos} is a full tripos, usually called localic tripos. A weak predicate classifier is given by $\Omega:=\loc$ and $\in:=\id_{\loc}$.
\end{example}
\begin{example}
The realizability doctrine $\doctrine{\set}{\mathcal{P}}$  defined in \Cref{ex:realizability tripos} is a full tripos. A weak predicate classifier is given by $\Omega:=\powerset(\pca{A})$ and $\in:=[\id_{\powerset(\pca{A})}]$.
\end{example}
\begin{example}\label{ex:subobject on topos}
If $\mE$ is a topos, then the functor $\doctrine{\mE}{\Sub}$, as defined in \Cref{ex:sub doctrine}, is a tripos.
\end{example}
\begin{example}\label{ex:slice of a tripos is a tripos}
    If $\doctrine{\mC}{P}$ is a lex primary doctrine with a weak predicate classifier given by an object $\Omega$ and $\in$ element of $P(\Omega)$ then the slice doctrine $\doctrine{\mC/X}{P_{/X}}$ as defined in \Cref{ex:slice doctrine} has a weak predicate classifier, given by the object $\freccia{\Omega\times X}{\pr_X}{X}$ and the element $P_{\pr_{\Omega}}(\in)$. Similarly, if $P$ is a full tripos then $P_{/X}$ is a full tripos.
\end{example} 

To guarantee that a doctrine of weak subobjects of a finite limit category is a tripos 
it is enough to check that it  has weak dependent products  in the sense of 
\cite{BirkedalCarboniRosolini}
and generic proof in the sense of \cite{MENNI2003287}. We recall the main definitions and briefly outline the key structures of this functor.
\begin{definition}\label{def: weak dependent products}
	A \textbf{weak dependent product} of an arrow $\freccia{X}{f}{J}$ of $\mC$ along a map $\freccia{J}{g}{I}$ consists of a commutative diagram

\[\begin{tikzcd}
	X & E & Z \\
	& J & I
	\arrow["e"', from=1-2, to=1-1]
	\arrow["f"', from=1-1, to=2-2]
	\arrow["g"', from=2-2, to=2-3]
	\arrow["h", from=1-3, to=2-3]
	\arrow["",from=1-2, to=2-2]
	\arrow[from=1-2, to=1-3]
	\arrow["\scalebox{1.6}{$\lrcorner$}"{anchor=center, pos=0.1}, shift left=3, draw=none, from=1-2, to=2-2]
\end{tikzcd}\]
	such that it has the universal property of being weakly terminal in the category of diagrams with such a shape, namely for every other diagram 
    \[\begin{tikzcd}
	X & E' & Z' \\
	& J & I
	\arrow["e'"', from=1-2, to=1-1]
	\arrow["f"', from=1-1, to=2-2]
	\arrow["g"', from=2-2, to=2-3]
	\arrow["h'", from=1-3, to=2-3]
	\arrow["",from=1-2, to=2-2]
	\arrow[from=1-2, to=1-3]
	\arrow["\scalebox{1.6}{$\lrcorner$}"{anchor=center, pos=0.1}, shift left=3, draw=none, from=1-2, to=2-2]
\end{tikzcd}\]
there are arrows $w: Z'\rightarrow Z$ and $k: E'\rightarrow E$ such that $e'=ek$ and $h'=hw$.
\end{definition}

    \begin{remark}\label{rem:weak dep prod implies slice weak dep prod}
    We recall from \cite{EMMENEGGER} that if a category has finite limits and weak dependent products then every slice category has weak dependent products.
\end{remark}

It is a known fact that if a category $\mC$ has weak dependent products, then the doctrine $\doctrine{\mC}{\Psi_{\mC}}$ is universal and implicational. The proof is straightforward, and the main idea is that we can define the universal quantifier $\freccia{Z}{\forall_g(f)}{I}$ as the weak dependent product of $f$ along $g$
%
and then define the implication as
$f\to g:= \forall_f \Psi_f (g)$.

Now we recall from \cite{MENNI2007511,MENNI2002187} the notion of \emph{generic proof}.
	
	\begin{definition}
	A \textbf{generic proof} is a morphism $\freccia{\Theta}{\theta}{\Lambda}$ such that for every arrow $\freccia{Y}{f}{X}$ there exists a map $\freccia{X}{\upsilon_f}{\Lambda}$ such that $f$ factors through $\upsilon_f^*\theta$ and $\upsilon_f^*\theta$  factors through $f$.
	\[\begin{tikzcd}
		Y & E  & \Theta \\
		& X & \Lambda
		\arrow["{\upsilon_f}"', from=2-2, to=2-3]
		\arrow["", from=1-2, to=2-2]
		\arrow["\theta", from=1-3, to=2-3]
		\arrow[from=1-2, to=1-3]
		\arrow["f"', from=1-1, to=2-2]
		\arrow["\scalebox{1.6}{$\lrcorner$}"{anchor=center, pos=0.1}, shift left=3, draw=none, from=1-2, to=2-2]
		\arrow["{e_1}"', shift right=1, from=1-1, to=1-2]
		\arrow["{e_2}"', shift right=1, from=1-2, to=1-1]
	\end{tikzcd}\]
\end{definition}
It is immediate to check that the doctrine $\doctrine{\mC}{\Psi_{\mC}}$ has a weak predicate classifier if and only if $\mC$ has a generic proof. 

\begin{remark}\label{rem:generic proof in slices}
    Notice that if $\mC$ has a generic proof  $\freccia{\Theta}{\theta}{\Lambda}$, it is straightforward to check that every slice category $\mC/X$ has a generic proof  given by 
\[\begin{tikzcd}
	{X\times \Theta} && {X\times \Lambda} \\
	& X
	\arrow["{\id_X\times \theta}", from=1-1, to=1-3]
	\arrow["{\pi_X}"', from=1-1, to=2-2]
	\arrow["{\pi_X}", from=1-3, to=2-2]
\end{tikzcd}\]
\end{remark}

\subsection{Tripos-to-topos construction}\label{sec:T2T}

Now we recall the tripos-to-topos construction \cite{TTT} in the general context of full existential doctrines \cite{UEC,TECH}.

\bigskip
\noindent
\textbf{Tripos-to-topos.} Given a full existential doctrine $\doctrine{\mC}{P}$, the category $\sT_P$ is defined as follows:
 
\medskip
\noindent 
\textbf{objects:} pairs $(A,\rho)$ such that $\rho \in P(A\times A)$  satisfies

\begin{itemize}

\item \emph{symmetry:} $\rho\leq P_{\angbr{\pr_2}{\pr_1}}(\rho)$;
\item \emph{transitivity:} $P_{\angbr{\pr_1}{\pr_2}}(\rho)\wedge P_{\angbr{\pr_2}{\pr_3}}(\rho)\leq P_{\angbr{\pr_1}{\pr_3}}(\rho)$, where $\pr_i$ are the projections from $A\times A\times A$;
\end{itemize} 
\textbf{arrows:} $\freccia{(A,\rho)}{\phi}{(B,\sigma)}$ are objects $\phi\in P(A\times B)$ such that
\begin{enumerate}[label=(\roman*)]
\item $\phi\leq P_{\angbr{\pr_1}{\pr_1}}(\rho)\wedge P_{\angbr{\pr_2}{\pr_2}}(\sigma)$;
\item $P_{\angbr{\pr_1}{\pr_2}}(\rho)\wedge P_{\angbr{\pr_1}{\pr_3}}(\phi)\leq P_{\angbr{\pr_2}{\pr_3}}(\phi)$ where $\pr_i$ are projections from the object $A\times A\times B$; 
\item $P_{\angbr{\pr_2}{\pr_3}}(\sigma)\wedge P_{\angbr{\pr_1}{\pr_2}}(\phi)\leq P_{\angbr{\pr_1}{\pr_3}}(\phi)$ where $\pr_i$ are projections from the object $A\times B\times B$; 
\item $P_{\angbr{\pr_1}{\pr_2}}(\phi)\wedge P_{\angbr{\pr_1}{\pr_3}}(\phi)\leq P_{\angbr{\pr_2}{\pr_3}}(\sigma)$ where $\pr_i$ are projections from the object $A\times B\times B$; 
\item $P_{\Delta_A}(\rho)\leq \exists_{\pr_1}(\phi)$ where $\pr_1$ is the first projection from $A\times B$.
\end{enumerate}
\medskip
\noindent
The following result is presented in \cite[Cor. 3.4]{UEC}:

\begin{theorem}
	Let $\doctrine{\mC}{P}$ be a full existential doctrine. Then $\sT_P$ is an exact category.
\end{theorem}

The following result is presented in \cite[Thm. 3.4]{TTT}:
\begin{theorem}
	Let $\doctrine{\mC}{P}$ be a full first-order hyperdoctrine. Then $\sT_P$ is a logos.
\end{theorem}

A necessary and sufficient condition for making $\sT_P$ a topos is provided by a careful analysis in \cite{TTT}. In particular, given a first-order hyperdoctrine $\doctrine{\mC}{P}$, we have that $\sT_P$ is a topos precisely when $P$ satisfies in its internal language a form of \emph{Comprehension Axiom}:
\begin{axiom}[CA]
A first-order hyperdoctrine $\doctrine{\mC}{P}$ satisfies the \textbf{Comprehension Axiom} if for any object $X$ of $\mC$ there exists an object $\mathrm{P}(X)$ and an element $\in_X$ of $P(X\times \mathrm{P}(X))$ such that for every object $Y$ of $\mC$ and every $\alpha\in P(X\times Y)$, $P$ satisfies the following sentence
   \[ \forall i: Y. \; \exists s: \mathrm{P}(X).\;\forall x: X. \;\in_X (x,s)\longleftrightarrow \alpha (x,i)\]
   in its internal language.
\end{axiom}
Hence, we have the following result, see \cite[Thm. 4.2]{TTT}.
\begin{theorem}\label{thm:hyperdoctrine and CA iff T topos}
Let $\doctrine{\mC}{P}$ be a first-order hyperdoctrine. Then $\sT_P$ is a topos if and only if $P$ satisfies (CA).
\end{theorem}
Since every tripos satisfies (CA) we have the following corollary:
\begin{cor}\label{cor:tripos-to-tops of a tripos is topos}
If $\doctrine{\mC}{P}$ is a tripos, then $\sT_P$ is a topos.
\end{cor}
The following examples are discussed in \cite{TT}, and show that realizability and localic toposes can be presented as instances of the same construction, i.e. the tripos-to-topos one.
\begin{example}\label{ex:loc_sh_as_ttt}
The topos $\sT_{\loc^{(-)}}$ associated to the tripos $\doctrine{\set}{\loc^{(-)}}$ is equivalent to the category of sheaves $\sh{\loc}$ of the locale.
\end{example}
\begin{example}
Given a pca $\pca{A}$, the tripos-to-topos $\sT_{\mathcal{P}}$ of the realizability tripos $\doctrine{\set}{\mathcal{P}}$ is equivalent to the realizability topos $\mathsf{RT}(\pca{A})$.
\end{example}
	
	\begin{example}\label{ex: exact comp of weak sub 1}  
	By \Cref{cor: regular come  ex lex}, we have that the exact completion  $\exlex{\mC}$ of a category $\mC$  with finite limits turns out to be equivalent to the tripos-to-topos $\sT_{\Psi_{\mC}}$ of the doctrine $\doctrine{\mC}{\Psi_{\mC}}$ of weak subobjects of $\mC$.
	\end{example}
The notion of \emph{geometric morphism} of triposes was originally introduced in \cite[Def. 3.4]{TT} for $\set$-based triposes. Here we consider a slight generalization for arbitrary based triposes that can be found in \cite{FREY2015232}.

\begin{definition}
Let $\doctrine{\mC}{P}$ and $\doctrine{\mD}{R}$ be two full triposes. 
  A \textbf{geometric morphism} of triposes $\freccia{P}{(F,\mathfrak{b})}{R}$ is a lex primary morphism of doctrines that has a left adjoint. A geometric morphism is said to be a \textbf{geometric embedding} if the counit of the adjunction is an iso.
\end{definition}

\begin{example}
    Assuming the axiom of choice, we have a geometric inclusion of the tripos $\doctrine{\set}{\Sub_{\set}}$ into the realizability tripos $\doctrine{\set}{\mathcal{P}}$ sending a mono $\freccia{X}{m}{Y}$ into $\exists_{m}(\top_X)$.
\end{example}
Now we recall a useful result, showing that a geometric morphism (and also a geometric inclusion) of triposes induces a geometric morphism (and a geometric inclusion) between their tripos-to-topos. This fact was already proved in the original work \cite[Prop. 3.5]{TT} for $\set$-based triposes, and it has been analyzed in the more general setting of arbitrary based triposes by Frey in \cite{FREY2015232}.
\begin{theorem}\label{cor geometric morphism of triposes induce geometric morphisms}
Every geometric morphism of triposes $(F,\mathfrak{b})\dashv (G,\mathfrak{c})$ induces a geometric morphism $\sT(F,\mathfrak{b})\dashv \sT(G,\mathfrak{c})$  of toposes between their tripos-to-topos construction. The same result holds for geometric inclusions.
\end{theorem}
\begin{remark}
    Notice that the left adjoint of a geometric morphism of triposes is a morphism of full existential doctrines.
\end{remark}

\section{Abstracting localic presheaves for lex primary doctrines}\label{Sec:ab loc presh}
The main purpose of this section is to generalize the notion of \emph{localic presheaves} to doctrines. For this purpose, we recall from \cite{CLTT,EQC,UEC} that:
\begin{definition}
The \textbf{Grothendieck construction} or \textbf{category of elements} $\mG_P$ of a lex primary doctrine $\doctrine{\mC}{P}$ is defined as follows:
\begin{itemize}
\item objects of $\mG_P$  are pairs $(A,\alpha)$, where $A$ is an object of $\mC$ and $\alpha\in P(A)$;
\item  a morphism  $\freccia{(A,\alpha)}{f}{(B,\beta)}$ in  $\mG_P$ is an arrow $\freccia{A}{f}{B}$ of $\mC$ such that $\alpha\leq P_f(\beta)$.
\end{itemize}
\end{definition}
It is straightforward to check that the category of elements of a lex primary doctrine inherits finite limits from the base category of the doctrine:
\begin{proposition}\label{prop: G_P has finite limits}
Let $\doctrine{\mC}{P}$ be a lex primary doctrine. Then $\mG_P$ has finite limits. 
\end{proposition}

\begin{example}\label{ex:A coprod comp}
Let $\loc$ be a locale. The category of elements $\mG_{\loc^{(-)}}$ of the localic tripos $\doctrine{\set}{\loc^{(-)}}$ is the category of \emph{fuzzy sets} \cite{Barr_1986}. Moreover this is equivalent to the coproduct completion $\loc_+$  of the locale $\loc$.
\end{example}

\begin{example}\label{ex:partition assemblies as grothendieck cat}
The category $\parassembly{\pca{A}}$ of partitioned assemblies (see e.g. \cite{van_Oosten_realizability,colimitcomprobinson}) of a partial combinatory algebra $\pca{A}$ is equivalent to the  category of elements $\mG_{\pca{A}^{(-)}}$ of the doctrine $\doctrine{\set}{\pca{A}^{(-)}}$ defined in \Cref{ex: primary doctrine pca}.
\end{example}

Now the main idea for generalizing the construction of the category of localic presheaves is to look at its categorical properties in terms of the generating triposes.
In particular, recall from
\cite{SFEC,MENNI2002187,MENNI2007511} that
the category of presheaves $\mathsf{PSh}(\loc)$  of a locale  $\loc$ can be equivalently described as the {\it exact-on-lex completion} of the
coproduct completion $\loc_+$ of the locale $\loc$
\[\mathsf{PSh}(\loc)\equiv \exlex{\loc_+}\]
or, even better, as
\[\mathsf{PSh}(\loc)\equiv \exlex{\mG_{\loc^{(-)}}}.\]
since the coproduct completion $\loc_+$ is given by the category of elements $\mG_{\loc^{(-)}}$ of the localic tripos $\doctrine{\set}{\loc^{(-)}}$ as observed in \Cref{ex:A coprod comp}.

This crucial observation suggests how to abstract the notion of category of localic presheaves to arbitrary doctrines:

\begin{definition}\label{def:P-presheaves}
Let $\doctrine{\mC}{P}$ be a lex primary doctrine. We define the \textbf{exact category of elements of} $P$ as the category $\presh{P}:=\exlex{\mG_P}$.
\end{definition}

\begin{example}\label{ex;localic presh topos}
 Let $\loc$ be a locale and the localic tripos $\doctrine{\set}{\loc^{(-)}}$. As anticipated at the beginning of this section, we have the equivalence $\presh{\loc^{(-)}}= \exlex{\loc_+}\equiv\mathsf{PSh}(\loc)$.
\end{example}
\begin{example}
    The category $\presh{\top}$ associated with the trivial lex primary doctrine $\doctrine{\mC}{\top}$ defined in \Cref{ex:weak sub is full ex comp} is precisely the exact completion $\exlex{\mC}$.
\end{example}

\begin{example}
By \Cref{ex:partition assemblies as grothendieck cat}, we have that the category $\parassembly{\pca{A}}$ of partitioned assemblies  of a given $\pca{A}$ is equivalent to the  category of elements $\mG_{\pca{A}^{(-)}}$ of the doctrine $\doctrine{\set}{\pca{A}^{(-)}}$ defined in \Cref{ex: primary doctrine pca}. Hence, the category $\presh{\pca{A}^{(-)}}$ associated with the doctrine $\doctrine{\set}{\pca{A}^{(-)}}$ is precisely $\presh{\pca{A}^{(-)}}=\exlex{\parassembly{\pca{A}}}\equiv \mathsf{RT}(\pca{A})$.

\end{example}

\begin{example}\label{ex:G_P realizability}
Let $\pca{A}$ be a pca, and let us consider the realizability tripos  $\doctrine{\set}{\mathcal{P}}$. Objects and morphisms of the category of elements $\mG_{\mathcal{P}}$ can be described as follows: they are pairs $(X,\alpha)$, where $X$ is a set and $\alpha\subseteq X\times \pca{A}$ is a relation. A morphism $\freccia{(X,\alpha)}{f}{(Y,\beta)}$ is given by a function $\freccia{X}{f}{Y}$ such that there exists an element $a\in \pca{A}$ that \emph{tracks} $f$, namely for every $x$ in $X$ and for every $b$ in $\pca{A}$, we have that if $x\alpha b$ then $f(x)\beta (a\cdot b)$.

\end{example}

\section{Abstracting localic supercompactification for lex primary doctrines}\label{Sec: ab localic supercomp}
The full existential completion of doctrines studied in \cite{ECRT,MaiettiTrotta21,FREY2023} is a construction that freely adds existential quantifiers along all the arrows of the base category to a given lex primary doctrine, which can be seen as the fibrewise poset reflection of the Family construction for fibrations, see e.g.~\cite{StreetFamily,hofstra2011} and~\cite[Sec.~B.1.4, Prop.~1.4.16]{SAE}. As a particular instance, it captures the construction of a supercoherent locale from an inf-semilattice in the sense of~\cite{BANASCHEWSKI199145}. 

This result, proved in~\cite[Sec.~7.4]{MaiettiTrotta21}, motivates identifying this construction as the doctrinal abstraction of the notion of \emph{supercompactification}.

We briefly recall the definition of the full existential completion from~\cite{ECRT} and its application to the construction of a supercoherent locale.

\bigskip
\noindent
\textbf{Full existential completion}. Let $\doctrine{\mC}{P}$ be a lex primary doctrine.  For every object $A$ of $\mC$ consider the following preorder:
\begin{itemize}
\item \textbf{objects:} pairs $(\frecciasopra{B}{f}{A},\alpha)$, where $A$ and $B$ are objects of $\mC$ and $\alpha\in P(B)$.
\item \textbf{order:}  $(\frecciasopra{B}{f}{A},\alpha)\leq (\frecciasopra{C}{g}{A},\beta)$ if there exists an arrow $\freccia{B}{h}{C}$ of $\mC$ such that $f=gh$ and $ \alpha\leq P_{h}(\beta).$

\end{itemize}
It is straightforward to check that the previous data give a preorder, and we denote by
$\compexfull{P}(A)$ the partial order obtained by its poset reflection.

Given a morphism $\freccia{A}{f}{B}$ in $\mC$, let $\compexfull{P}_f(\frecciasopra{C}{g}{B},\beta)$ be the object 
$(\frecciasopra{D}{\pbmorph{f}{g}}{A},\; P_{g^*f}(\beta) )$
where  $f^*g$ and $g^*f$ are defined by the pullback of $f$ and $g$ as usual.

The assignment $\doctrine{\mC}{\compexfull{P}}$ is called the \textbf{full existential completion} of $P$.

\begin{theorem}\label{thm: compex P is existential and RC}
The doctrine $\doctrine{\mC}{\compexfull{P}}$ is a full existential doctrine. Moreover the assignment $P\mapsto \compexfull{P}$ extends to a 2-functor 

\[\begin{tikzcd}
		\fullprDoc && \EED
		\arrow[""{name=0, anchor=center, inner sep=0}, "\compexfull{(-)}",curve={height=-14pt}, from=1-1, to=1-3]
		\arrow[""{name=1, anchor=center, inner sep=0}, "",curve={height=-14pt}, hook',from=1-3, to=1-1]
		\arrow["\dashv"{anchor=center, rotate=-90}, draw=none, from=1, to=0]
	\end{tikzcd}\]
 from the 2-category $\fullprDoc$ of lex primary doctrines to the 2-category $\EED$ of full existential doctrines that is left adjoint to the inclusion 2-functor.
\end{theorem}
We define the canonical inclusion
 $$\freccia{P}{(\id_{\mC},\mathfrak{i})}{\compexfull{P}}$$
 where $\freccia{P(A)}{\mathfrak{i}}{\compexfull{P}(A)}$ sends $\alpha\mapsto (\frecciasopra{A}{\id_A}{A}, \alpha)$. Notice that $(\id_{\mC},\mathfrak{i})$ is a lex primary morphism of doctrines, but it does not preserve existential quantifiers, i.e. it is not full existential, in general. We recall also that the full existential completion is not an idempotent operation in general, and that $P\cong P^{\exists}$ holds if and only if $P$ is equipped with a generalization of Hilbert's epsilon operators, see \cite{MaiettiTrotta21}.
As a direct application of the universal property of the full existential completion we have the following corollary.
\begin{cor}\label{cor:adjunction}
Let $\doctrine{\mC}{P}$ be a full existential doctrine. Then there exists a full existential morphism $\freccia{\compexfull{P}}{(\id_{\mC},\bar{\mathfrak{i}})}{P}$ such that $(\id_{\mC},\bar{\mathfrak{i}})\dashv (\id_{\mC},\mathfrak{i})$ and $(\id_{\mC},\bar{\mathfrak{i}})(\id_{\mC},\mathfrak{i})\cong \id_P$. 
\end{cor}

\begin{cor}\label{prop:useful_adjunction}
Let $\doctrine{\mC}{P}$ be a lex primary doctrine. Then there is an adjunction of full existential doctrines whose counit is an iso 
\[\begin{tikzcd}
	\mC^{\op} \\
	&& \infsl \\
	\mC^{\op}
	\arrow["\id_{\mC}^{\op}"',from=1-1, to=3-1]
	\arrow[""{name=0, anchor=center, inner sep=0},"\compexfull{P}"', from=3-1, to=2-3]
	\arrow[""{name=1, anchor=center, inner sep=0}, "\Psi_{\mC}",from=1-1, to=2-3]
	\arrow[""{name=2, anchor=center, inner sep=0},"\bar{\mathfrak{b}}", curve={height=-14pt}, shorten <=5pt, shorten >=5pt, from=0, to=1]
	\arrow[""{name=3, anchor=center, inner sep=0},"{\mathfrak{b}}", curve={height=-14pt}, shorten <=6pt, shorten >=6pt, from=1, to=0]
	\arrow["\dashv"{anchor=center}, draw=none, from=2, to=3]
\end{tikzcd}\]
\end{cor}
\begin{proof}
    It is enough to observe that for every lex primary doctrine $P$ we have an adjunction of lex primary doctrines 

    \[\begin{tikzcd}
	\mC^{\op} \\
	&& \infsl \\
	\mC^{\op}
	\arrow["\id_{\mC}^{\op}"',from=1-1, to=3-1]
	\arrow[""{name=0, anchor=center, inner sep=0},"P"', from=3-1, to=2-3]
	\arrow[""{name=1, anchor=center, inner sep=0}, "\top",from=1-1, to=2-3]
	\arrow[""{name=2, anchor=center, inner sep=0},"\bar{\mathfrak{t}}", curve={height=-14pt}, shorten <=5pt, shorten >=5pt, from=0, to=1]
	\arrow[""{name=3, anchor=center, inner sep=0},"{\mathfrak{t}}", curve={height=-14pt}, shorten <=6pt, shorten >=6pt, from=1, to=0]
	\arrow["\dashv"{anchor=center}, draw=none, from=2, to=3]
\end{tikzcd}\]
where $\mathfrak{t}$ sends the top element into the top, and $\bar{\mathfrak{t}}$ sends every element of the fibre of $P$ into the unique element of the fibre of $\top$. Hence, when we apply the full existential completion to each component of the diagram we obtain the desired adjunction, since by \Cref{ex:weak sub is full ex comp}, the full existential completion of the trivial doctrine is the weak subobject doctrine, i.e.  $\compexfull{\top}=\Psi_{\mC}$.
\end{proof}

Our main examples  of full existential completions are the weak subobject doctrine,  realizability triposes and localic triposes associated with a supercoherent locale. We refer to  \cite{MaiettiTrotta21} and \cite{FREY2023} for more details:
\begin{example}\label{ex:weak sub is full ex comp}
Let $\mC$ be a category with finite limits. Then the weak subobject doctrine $\doctrine{\mC}{\Psi_{\mC}}$ is the full existential completion of the trivial doctrine on $\doctrine{\mC}{\top}$ which sends every object $A$ into the poset $\top(A):=\{\top\}$ with a single element.
\end{example}

\begin{example}\label{ex: real tripos is an existential comp}
Given a pca $\pca{A}$, the realizability doctrine $\doctrine{\set}{\mathcal{P}}$ defined in \Cref{ex:realizability tripos} is the full existential completion of the lex primary doctrine $\doctrine{\set}{\pca{A}^{(-)}}$  defined in \Cref{ex: primary doctrine pca}.
\end{example}
While all realizability triposes are   full existential completions, not all localic triposes are so. We refer to \cite{MaiettiTrotta21} for more details.
  \begin{theorem}\label{ex:supercoherent localic doctrine is full ex comp}
	The localic doctrine $\doctrine{\set}{\loc^{(-)}}$ is a full existential completion if and only if $\loc$ is supercoherent in the sense of \cite{BANASCHEWSKI199145}. In particular, when $\loc$ is supercoherent, the localic tripos $\doctrine{\set}{\loc^{(-)}}$ is the full existential completion of the lex primary doctrine  $\doctrine{\set}{\supercomp^{(-)}}$ associated to the inf-semilattice $\supercomp$ of the supercompact elements of $\loc$. In particular, we have the following isomorphisms of doctrines
	\[\compexfull{(\supercomp^{(-)})}\cong D(\supercomp)^{(-)}\cong\loc^{(-)}\]
	where $D(\supercomp)$ is the supercoherent locale generated by $\supercomp$.
  \end{theorem}
Now we provide a result that will be useful in the following sections. First we recall the following \emph{rule of choice} from \cite{MaiettiRosoliniRelating,TECH}.
\begin{Rule}[RC]\label{rule of choice}
A full existential doctrine $\doctrine{\mC}{P}$ satisfies the \textbf{Rule of Choice} when, for every element $\alpha$ of $P(A\times B)$, if $\top_A\leq \exists_{\pr_A}(\alpha)$ then there exists an arrow $\freccia{A}{f}{B}$ of $\mC$ such that $\top_A\leq P_{\angbr{\id_A}{f}}(\alpha)$;
\end{Rule}
We recall the following result from \cite[Thm. 4.16]{MaiettiTrotta21}:
\begin{proposition}\label{pro:comp_ex_sat_RC}
    Let $\doctrine{\mC}{P}$ be a lex primary doctrine. Then $\compexfull{P}$ satisfies the Rule of Choice.
\end{proposition}
\begin{example}\label{ex:weak sub rule of choice}
   The weak subobject doctrine $\doctrine{\mC}{\Psi_{\mC}}$ satisfies the rule of choice, see \cite{TECH}.
\end{example}
\begin{example}
	The realizability tripos $\doctrine{\set}{\mathcal{P}}$ satisfies the rule of choice, see \cite{MaiettiTrotta21}.
\end{example}
Now we present a useful characterization of first-order hyperdoctrines satisfying the rule of choice: 
\begin{proposition}\label{prop: CA+RC implies tripos}
Let $\doctrine{\mC}{P}$ be a first-order hyperdoctrine such that:
\begin{enumerate}
\item $P$ satisfies the Comprehension Axiom (CA);
\item $P$ satisfies the Rule of Choice (RC). 

\end{enumerate}
Then $\doctrine{\mC}{P}$ is a tripos.
\end{proposition}
\begin{proof}
Let $\alpha$ be an element of the fibre $P(X\times Y)$. Since $P$ satisfies (CA), we have that the inequality
   \[ \top_1\leq \forall_{!_Y} \exists_{\pr_Y}\forall_{\angbr{\pr_{ \mathrm{P}(X)}}{\pr_Y}}(P_{\angbr{\pr_X}{\pr_{\mathrm{P}(X)}}}(\in_X) \longleftrightarrow P_{\angbr{\pr_X}{\pr_Y}}(\alpha) )\]
 holds in $P(1)$, where the domain of the projections is $X\times \mathrm{P}(X)\times Y$. Hence the inequality
   \[ \top_Y\leq \exists_{\pr_Y}\forall_{\angbr{\pr_{ \mathrm{P}(X)}}{\pr_Y}}(P_{\angbr{\pr_X}{\pr_{\mathrm{P}(X)}}}(\in_X) \longleftrightarrow P_{\angbr{\pr_X}{\pr_Y}}(\alpha) )\]
holds in $P(Y)$. Now, since $P$ satisfies the Rule of Choice we have that there exists an arrow $\freccia{\mathrm{Y}}{f}{\mathrm{P}(X)}$ such that

\[ \top_Y\leq P_{\angbr{f}{\id_Y}}\forall_{\angbr{\pr_{ \mathrm{P}(X)}}{\pr_Y}}(P_{\angbr{\pr_X}{\pr_{\mathrm{P}(X)}}}(\in_X) \longleftrightarrow P_{\angbr{\pr_X}{\pr_Y}}(\alpha) ).\]
By (BCC), we have that
\[ \top_Y\leq \forall_{\pr_Y} P_{\angbr{\pr_X,f\pr_Y}{\pr_Y}}(P_{\angbr{\pr_X}{\pr_{\mathrm{P}(X)}}}(\in_X) \longleftrightarrow P_{\angbr{\pr_X}{\pr_Y}}(\alpha) ).\]
Therefore
\[ \top_{X\times Y}\leq  P_{\angbr{\pr_X,f\pr_Y}{\pr_Y}}(P_{\angbr{\pr_X}{\pr_{\mathrm{P}(X)}}}(\in_X) \longleftrightarrow P_{\angbr{\pr_X}{\pr_Y}}(\alpha) )\]
in $P(X\times Y)$.
Hence
\[\top_{X\times Y}\leq P_{\id_X\times f} (\in_X) \longleftrightarrow \alpha.\]
Thus, we can conclude that for every element $\alpha$ of $P(X\times Y)$ there exists an arrow $\freccia{Y}{\{\alpha\}:=f}{\mathrm{P}(X)}$ such that  $P_{\id_X\times \{\alpha\}} (\in_X) = \alpha$, i.e. that $P$ is a tripos.
\end{proof}
Combining this result with \Cref{thm:hyperdoctrine and CA iff T topos} we obtain the following useful corollary:
\begin{cor}\label{cor:P_ex_is_tripos_iff_its_ttt_is_ topos}
    Let us suppose that $\doctrine{\mC}{P^{\exists}}$ is a first-order hyperdoctrine. Then the category $\sT_{P^{\exists}}$ is a topos if and only if $P^{\exists}$ is a tripos.
\end{cor}

\section{Abstracting Localic Comparison Lemma for lex primary doctrines}\label{sec:com lemma}
The main purpose of this section is to show that the equivalence
\[\sh{D(\loc)}\equiv \mathsf{PSh}(\loc)\]
obtained by applying the comparison lemma to the categories $D(\loc)$ and $\loc$,  equipped with a suitable coverage, see e.g. \cite[Ex. 2.2.4 (d)]{SAE2}, can be generalized to an equivalence
\[ \sT_{P^{\exists}}\equiv \presh{P}.\]
To reach this goal, we need to recall and characterize the regular and exact completions of  doctrines.

\subsection{Preliminaries on the regular and exact completion of a doctrine}
We recall from \cite{UEC,TECH} the construction of a \emph{free regular category} from a full existential doctrine. This construction can be applied in the more general context of elementary and existential doctrines, but for the purpose of this work, we just need to recall the construction for full existential doctrines: 
\begin{definition}\label{def: Reg P}
	Let $\doctrine{\mC}{P}$ be a full existential doctrine. We define the category $\regdoctrine{P}$ as follows:
\begin{itemize}
	\item the \textbf{objects} of $\regdoctrine{P}$ are pairs $(A,\alpha)$, where $A$ is an object of $\mC$ and $\alpha$ is an element of $P(A)$;
	\item an \textbf{arrow} of $\regdoctrine{P}$ from $(A,\alpha)$ to $(B,\beta)$ is given by an element $\phi$ of $P(A\times B)$ such that:
	\begin{enumerate}
		\item $\phi \leq P_{\pr_A}(\alpha)\wedge P_{\pr_B}(\beta)$;
		\item $\alpha\leq \exists_{\pr_A}(\phi)$;
		\item $P_{\angbr{\pr_1}{\pr_2}}(\phi)\wedge P_{\angbr{\pr_1}{\pr_3}}(\phi)\leq P_{\angbr{\pr_2}{\pr_3}}(\delta_B)$.
	\end{enumerate}
	\end{itemize}

	The composition of morphisms of $\regdoctrine{P}$ is given by the usual \emph{relational composition}: the composition of $\freccia{(A,\alpha)}{\phi}{(B,\beta)}$ and $\freccia{(B,\beta)}{\psi}{(C,\gamma)}$ is given by 
	\[ \exists_{\angbr{\pr_1}{\pr_3}}(P_{\angbr{\pr_1}{\pr_2}}(\phi)\wedge  P_{\angbr{\pr_2}{\pr_3}}(\psi))\]
	where $\pr_i$ are projections from $A\times B\times C$.
\end{definition}
\begin{definition}\label{def:regular completion full ex doctrine}
	Let $\doctrine{\mC}{P}$ be a full existential doctrine. 
	We call the category $\regdoctrine{P}$  the  \textbf{regular completion}  of $P$.
	\end{definition}
	 The choice of the name in \Cref{def:regular completion full ex doctrine} is justified by the following result, which is a specialization of \cite[Thm. 3.3]{TECH} for the case of full existential doctrines:
	\begin{theorem}\label{thm: regular completion full ex doctrine}
	Let $\doctrine{\mC}{P}$ be a full existential doctrine. Then the category $\regdoctrine{P}$ is regular and the assignment $P\mapsto \regdoctrine{P}$ extends to a left biadjoint
	\[\begin{tikzcd}
		\EED && \Reg
		\arrow[""{name=0, anchor=center, inner sep=0}, "\regdoctrine{-}",curve={height=-14pt}, from=1-1, to=1-3]
		\arrow[""{name=1, anchor=center, inner sep=0}, "\Sub_{(-)}",curve={height=-14pt}, hook',from=1-3, to=1-1]
		\arrow["\dashv"{anchor=center, rotate=-90}, draw=none, from=1, to=0]
	\end{tikzcd}\]
	to the inclusion of the 2-category $\Reg$ of regular categories in the 2-category $\EED$ of full  existential doctrines acting as $\mC\mapsto \Sub_{\mC}$.
	\end{theorem}
	\begin{remark}
		The previous result has been further extended to the setting of \emph{regular sites} and \emph{existential doctrinal sites} by Wrigley in his PhD thesis, see \cite[Sec.III.3.3,Prop. III.38]{WrigleyPHD} for more details.
	\end{remark}
The exact completion of a full existential doctrine is obtained by combining the regular completion of a doctrine with the $\exreg{-}$ exact completion of a regular category introduced by Freyd in \cite{CA}. We recall from \cite{UEC}  the following result:
\begin{cor}\label{theorem maietti rosolini pasquali exact comp}
	The assignment $P\mapsto \exreg{\regdoctrine{P}}$ extends to a left biadjoint to the 2-functor $\Sub_{\mC}\colon \Excat\rightarrow \EED$ sending an exact category $\mC$ to doctrine $\Sub_{\mC}$ of its subobjects.
	\end{cor}

The following example appears in \cite{TECH}. 
\begin{example}\label{cor: regular come  ex lex}
The regular  completion  $\reglex{\mC}$ of a category with finite limits $\mC$ is equivalent to the regular completion $\regdoctrine{\Psi_{\mC}}$ of the lex primary  doctrine $\doctrine{\mC}{\Psi_{\mC}}$ of weak subobjects of $\mC$. Hence, the exact completion  $\exlex{\mC}$ of a category $\mC$  is equivalent to the exact completion of the doctrine $\doctrine{\mC}{\Psi_{\mC}}$ of weak subobjects of $\mC$. 
\end{example}

We recall now the link with the tripos-to-topos from  \cite[Sec. 2]{UEC}:
	\begin{theorem}\label{teorema T_P=reg(P)_exreg}
	 For any  full existential doctrine $\doctrine{\mC}{P}$, the category $\sT_P$ is exact and the following equivalence holds $\sT_{P}\equiv \exreg{\regdoctrine{P}}$.
	\end{theorem}

	\begin{cor}\label{cor funct of tripos to topos}
	Given a full existential doctrine $\doctrine{\mC}{P}$, the assignment  $P\mapsto \sT_P$ extends to a left biadjoint
	\[\begin{tikzcd}
		\EED && \Excat
		\arrow[""{name=0, anchor=center, inner sep=0}, "\sT_{(-)}",curve={height=-14pt}, from=1-1, to=1-3]
		\arrow[""{name=1, anchor=center, inner sep=0}, "\Sub_{(-)}",curve={height=-14pt}, hook',from=1-3, to=1-1]
		\arrow["\dashv"{anchor=center, rotate=-90}, draw=none, from=1, to=0]
	\end{tikzcd}\]
	 to the inclusion of the 2-category $\Excat$ of exact categories in the 2-category $\EED$ of full  existential doctrines acting as $\mC\mapsto \Sub_{\mC}$.
	\end{cor}

These results can be combined with Menni's characterization of lex categories whose exact completions are toposes \cite{MENNI2003287} to show that having weak dependent products and a generic proof are not only sufficient, but also necessary for making $\doctrine{\mC}{{\Psi_{\mC}}}$ a tripos:

\begin{theorem}[Menni \cite{MENNI2003287}]\label{cor: Menni chararcterisation}
	Let $\mC$ be a category with finite limits. Then $\exlex{\mC}$ is a topos if and only if $\mC$ has weak dependent products and a generic proof.
	\end{theorem}
	Combining this result with \Cref{ex: exact comp of weak sub 1} and \Cref{rem:weak cart closed power ob iff pred clas}, we obtain the following corollary:
	\begin{cor}\label{thm:weak suobject tripos}
		Let $\mC$ be a category with finite limits. Then the weak subobject doctrine $\doctrine{\mC}{{\Psi_{\mC}}}$ is a full tripos if and only if $\mC$ has weak dependent products and a generic proof.
		\end{cor}
	\paragraph{Graphical summary} Regular and exact completions of lex categories factor through the category of full existential doctrines, as shown below.

\[\begin{tikzcd}[column sep=huge, row sep=huge ]
	\mathsf{LexCat}  & \mathsf{RegCat} &\mathsf{ExCat} \\
	& \EED
	\arrow["\reglex{-}",squiggly, from=1-1, to=1-2]
	\arrow["\Psi_{(-)}"', hook, squiggly, from=1-1, to=2-2]
	\arrow["\regdoctrine{-}"',squiggly, from=2-2, to=1-2]
	\arrow["\exreg{-}",squiggly, from=1-2, to=1-3]
	\arrow["\sT_{(-)}"',squiggly, from=2-2, to=1-3]
\end{tikzcd}\]
\subsection{Abstracting the localic Comparison Lemma}

The first result we are going to show is that the regular completion of the full existential completion $\compexfull{P}$ of  a lex primary doctrine $\doctrine{\mC}{P}$ is equivalent to the  regular category associated to the category of elements $\mG_{P}$ of $P$
$$\regdoctrine{\compexfull{P}}\equiv \reglex{\mG_{P}}.$$

Observe now that there is a strong link between the fibres of a full existential completion of a doctrine $P$  and
the weak subobject doctrine of the Grothendieck construction of $P$.
For this purpose, let us denote by $\freccia{\mC}{I_{\mC}}{\mG_P}$ the canonical functor sending $A\mapsto I_{\mC}(A):=(A,\top)$.

\begin{lemma}\label{thm: full existential doctrine as weak subobject}
	Let $\doctrine{\mC}{P}$ be a lex primary doctrine. Then $\compexfull{P}	= \Psi_{\mG_{P}}\circ I_{\mC}^{\op}$.
	\end{lemma}

\begin{lemma}\label{lem: presentazione oggeti reglex Grothendieck}
	Let $\doctrine{\mC}{P}$ be a lex primary doctrine. Every object of $\reglex{\mG_P}$ is isomorphic to one of the form $\freccia{(A,\alpha)}{f}{(B,\top_B)}$.
	\end{lemma}
	
    \begin{proof}
    This follows from the definition of arrows in the $\mathsf{reg}/\mathsf{lex}$-completion \cite{REC,SFEC}. Indeed, if we consider an object $\frecciasopra{(A,\alpha)}{f}{(B,\beta)}$ of the category $\reglex{\mG_P}$, then it is straightforward to check that 
\[\begin{tikzcd}
	{(A,\alpha)} & {(B,\beta)} \\
	{(A,\alpha)} & {(B,\top)}
	\arrow[""{name=0, anchor=center, inner sep=0}, "f", from=1-1, to=1-2]
	\arrow[""{name=1, anchor=center, inner sep=0}, "f"', from=2-1, to=2-2]
	\arrow["{[\id]}", shorten <=4pt, shorten >=4pt, from=0, to=1]
\end{tikzcd}\]
	is an arrow of $\reglex{\mG_P}$ and that it is an isomorphism.
    
	\end{proof}
Combining \Cref{lem: presentazione oggeti reglex Grothendieck} with \Cref{cor: regular come  ex lex} we obtain the following corollary:
	\begin{cor}\label{cor: oggetti reg weak sub su grot. P}
		Let us consider the lex doctrine $\doctrine{\mG_P}{\Psi_{\mG_P}}$, where $\doctrine{\mC}{P}$ is a lex primary doctrine. Then every object $((B,\beta),f)$ of $\regdoctrine{\Psi_{\mG_P}}$ is isomorphic to one of the form  $((B,\top),f)$.
	\end{cor}
Now we observe that the regular completion of a full existential doctrine depends only on the objects of the base category and on the fibres:
\begin{lemma}\label{lemma P ha fibre isomorfe a R}
If $\freccia{P}{(F,f)}{R}$ is a morphism of full existential doctrines such that $\freccia{P(A)}{f_A}{R(FA)}$ is an isomorphism for every object $A$ of the base, then $\freccia{\regdoctrine{P}}{\regdoctrine{F,f}}{\regdoctrine{R}}$ is a full and faithful functor. 
\end{lemma}
\begin{proof}
Since $\freccia{P}{(F,f)}{R}$ is a morphism of full existential doctrines and $f$ is a natural transformation whose components $\freccia{P(A)}{f_A}{R(FA)}$ are isomorphisms we have that an element $\phi$ of the fibre $P(A\times B)$ provides a morphism $\freccia{(A,\alpha)}{\phi}{(B,\beta)}$ of $\regdoctrine{P}$ if and only if the element $f_{A\times B}(\phi)$ of the fibre $RF(A\times B)$ (that is isomorphic to the fibre $R(FA\times FB)$ since $F$ is a finite limits preserving functor) provides a morphism $\freccia{(FA,f_A(\alpha))}{f_{A\times B}(\phi)}{(FB,f_B(\beta))}$. Therefore, we have that the functor  $\freccia{\regdoctrine{P}}{\regdoctrine{F,f}}{\regdoctrine{R}}$ is full and faithful.
\end{proof}

\begin{theorem}\label{thm:caratterizzazione regular comp existential completion}
Let $\doctrine{\mC}{P}$ be a lex primary doctrine. Then 
\[\regdoctrine{\compexfull{P}}\equiv \reglex{\mG_P}\]
\end{theorem}
\begin{proof}
	By \Cref{thm: full existential doctrine as weak subobject} we have that $\regdoctrine{\compexfull{P}}\equiv  \regdoctrine{\Psi_{\mG_P}\circ I_{\mC}^{\op}}$. To conclude the proof it is enough to observe that by \Cref{lemma P ha fibre isomorfe a R} the category $\regdoctrine{\Psi_{\mG_P}\circ I_{\mC}^{\op}}$ is a full sub-category of $\regdoctrine{\Psi_{\mG_P}}$. In particular, by definition of $\regdoctrine{-}$ (see \Cref{def: Reg P}) it is the full subcategory whose objects are of the form $((B,\top),f)$. Hence we can apply \Cref{cor: oggetti reg weak sub su grot. P} to conclude that 
	\[\regdoctrine{\Psi_{\mG_P}\circ I_{\mC}^{\op}}\equiv \regdoctrine{\Psi_{\mG_P}}.\]
	Therefore, combining these equivalences with \Cref{cor: regular come  ex lex} we can conclude that 
	\[ \regdoctrine{\compexfull{P}}\equiv \regdoctrine{\Psi_{\mG_P}\circ I_{\mC}^{\op}} \equiv\regdoctrine{\Psi_{\mG_P}}\equiv \reglex{\mG_P}. \]
	
\end{proof}

	\begin{example}\label{ex: regular completion realizability tripos}
		By combining \Cref{thm:caratterizzazione regular comp existential completion} with \Cref{ex:partition assemblies as grothendieck cat} and \Cref{ex: real tripos is an existential comp} we obtain that the regular completion of the realizability tripos is equivalent to the category of assemblies:

		\[ \regdoctrine{ \mathcal{P}}\equiv \reglex{\mG_{\pca{A}^{(-)}}}\equiv \reglex{\parassembly{\pca{A}}}\]
        i.e. the category  $ \regdoctrine{ \mathcal{P}}$ is equivalent to the category of assemblies $\assembly{\pca{A}}$ \cite{van_Oosten_realizability,colimitcomprobinson}.
		\end{example}

Combining 
\Cref{teorema T_P=reg(P)_exreg} with \Cref{thm:caratterizzazione regular comp existential completion} we obtain our first main result, that will provide a generalization in the context of doctrines of the (canonical) localic instance of the Comparison Lemma:
\begin{theorem}\label{thm:characterisation tripostotopos full ex comp}
	Let $\doctrine{\mC}{P}$ be a lex primary doctrine. Then we have an equivalence
	\[\sT_{\compexfull{P}}\equiv \exlex{\mG_P}=\presh{P}.\]
	\end{theorem}
    \Cref{thm:characterisation tripostotopos full ex comp} and \Cref{ex;localic presh topos} show that the tripos-to-topos construction is able to capture as particular cases both localic sheaf and presheaf toposes.
                
 Combining the universal properties of the tripos-to-topos  and the full existential completion we can  conclude that the construction $\presh{(-)}$ has the universal property of being the \emph{exact completion of lex primary doctrines}:
                
            \begin{cor}[exact completion of lex primary doctrines]\label{thm:presh universal propery}
                   The category $\presh{P}$ is an exact category, and the assignment $P\mapsto \presh{P}$ extends to a left biadjoint
                   \[\begin{tikzcd}
					\fullprDoc && \Excat
					\arrow[""{name=0, anchor=center, inner sep=0}, "\presh{(-)}",curve={height=-14pt}, from=1-1, to=1-3]
					\arrow[""{name=1, anchor=center, inner sep=0}, "\Sub_{(-)}",curve={height=-14pt}, hook',from=1-3, to=1-1]
					\arrow["\dashv"{anchor=center, rotate=-90}, draw=none, from=1, to=0]
				\end{tikzcd}\]
                   to the 2-functor sending an exact category $\mC$ to doctrine $\Sub_{\mC}$ of its subobjects.
                \end{cor}

\begin{cor}\label{teorema prefasci e fasci}
                    Let $\loc$ be a locale. Then we have the following equivalence
                    \[\sh{D(\loc)}\equiv \mathsf{PSh}{(\loc)}\]
                    where $D(\loc)$ is the supercoherent locale constructed from $\loc$.
                 \end{cor}
      
    Combining \Cref{prop:useful_adjunction} with \Cref{thm:characterisation tripostotopos full ex comp} we have the following corollary, relating the exact completion of the base category of a lex primary doctrine with the tripos-to-topos of its full existential completion:
\begin{cor}\label{cor_C_exlex_adj_T_P_exists}
    Let $\doctrine{\mC}{P}$ be a lex primary doctrine. Then we have an adjunction of exact categories whose co-unit is an iso:
        \[\begin{tikzcd}
					\presh{P}\equiv \sT_{\compexfull{P}} && \exlex{\mC}
					\arrow[""{name=0, anchor=center, inner sep=0}, "",curve={height=-14pt}, from=1-1, to=1-3]
					\arrow[""{name=1, anchor=center, inner sep=0}, "",curve={height=-14pt}, hook',from=1-3, to=1-1]
					\arrow["\dashv"{anchor=center, rotate=-90}, draw=none, from=1, to=0]
				\end{tikzcd}\]
\end{cor}
		\begin{example}\label{Example realizability topos as exact com}
			A relevant topos having the presentation of \Cref{thm:characterisation tripostotopos full ex comp} is provided by the realizability topos $\mathsf{RT} (\pca{A})$ associated to a pca $\pca{A}$ \cite{TT,colimitcomprobinson,HYLAND1982165}. In particular, we recall from \cite{colimitcomprobinson} that we have an equivalence $\mathsf{RT}(\pca{A})\equiv \exlex{\parassembly{\pca{A}}}$ between a realizability topos and the exact completion of the lex category of partitioned assemblies. Given a realizability doctrine $\doctrine{\set}{\mathcal{P}}$ associated to a pca $\pca{A}$, we have
			that combining \Cref{thm:characterisation tripostotopos full ex comp} with \Cref{ex:partition assemblies as grothendieck cat} and \Cref{ex: real tripos is an existential comp} we can conclude that
			\[\sT_{\mathcal{P}}\equiv \exlex{\mG_{\pca{A}^{(-)}}}\equiv \exlex{\parassembly{\pca{A}}}\equiv \mathsf{RT}(\pca{A}).\]

			\end{example}

			\begin{example}\label{example topos of sheaves over supercomp locale}
				
				As observed in \cite[Ex. 2.15]{TT}, the topos $\sT_{\loc^{(-)}}$ obtained by the localic doctrine associated to a locale $\loc$ has to be equivalent to the category $\sheaves{\loc}$ of canonical sheaves over $\loc$. When $\loc$ is supercoherent then, by \Cref{ex:supercoherent localic doctrine is full ex comp}, we have that the localic doctrine $\doctrine{\set}{\loc^{(-)}}$ is the full existential completion of the doctrine  $\doctrine{\set}{\supercomp^{(-)}}$ associated to the inf-semilattice $\supercomp$ of the supercompact elements of $\loc$. 
				 Therefore, by \Cref{thm:characterisation tripostotopos full ex comp}, we have
				\[\sT_{\loc^{(-)}}\equiv \sheaves{\loc}\equiv \exlex{\mG_{\supercomp^{(-)}}}\]
				where $\supercomp$ is the inf-semilattice of the supercompact elements of $\loc$.
						
				\end{example}
			\paragraph{Graphical summary} The constructions of doctrines and categories presented in this section, together with their relationships, are summarized in the following diagram:

\[\begin{tikzcd}[column sep=huge, row sep=huge ]
	\mathsf{LexCat}  & \mathsf{RegCat} &\mathsf{ExCat} \\
	\fullprDoc & \EED
	\arrow["\reglex{-}",squiggly, from=1-1, to=1-2]
	\arrow["\mG_{(-)}",squiggly, from=2-1, to=1-1]
	\arrow["\regdoctrine{-}"',squiggly, from=2-2, to=1-2]
	\arrow["\exreg{-}",squiggly, from=1-2, to=1-3]
	\arrow["\sT_{(-)}"',squiggly, from=2-2, to=1-3]
	\arrow["\compexfull{(-)}"',squiggly, from=2-1, to=2-2]
	\arrow[curve={height=80pt},"\mathsf{E}_{(-)}"',squiggly,from=2-1, to=1-3]
\end{tikzcd}\]

\section{A characterization of abstract localic presheaf toposes}\label{Sec char abstr localic presh tops}

Observe that while the notion of a tripos is sufficient to ensure that its tripos-to-topos construction yields a topos, it does not, in general, guarantee that the full existential completion of a tripos is itself a tripos, and hence that the corresponding category $\presh{P}$ of a tripos $P$ is again a topos.

The underlying intuition is that the latter construction depends not only on the elements of the fibres of a tripos, but also on the arrows of its base category.

\begin{definition}[$\exists$-supercompactifiable doctrine]\label{def:exists-sheaf tripos}
A lex primary doctrine $\doctrine{\mC}{P}$ is called $\exists$-\textbf{supercompactifiable} if the category of elements $\mG_P$ has weak dependent products and a generic proof. We denote by $\exists\mbox{-}\fullprDoc$  the 2-full subcategory of $\fullprDoc$ whose objects are $\exists$-supercompactifiable doctrines.
\end{definition}
Employing our previous analysis, we can provide the following characterization of $\exists$-supercompactifiable doctrines.
\begin{theorem}\label{thm:comp ex tripos is equivalent to}
   Let $\doctrine{\mC}{P}$ be a lex primary doctrine. Then the following are equivalent:
   \begin{enumerate}
      \item  $\doctrine{\mC}{P}$ is an $\exists$-supercompactifiable doctrine;
      \item  $\doctrine{\mG_P}{\Psi_{\mG_P}}$ is a full tripos;
      \item  $\doctrine{\mC}{\compexfull{P}}$ is a full tripos;
      \item $\presh{P}$ is a topos.
   \end{enumerate}
   \end{theorem}
   \begin{proof}
      $\it{(1\Rightarrow 2)}$ It follows by \Cref{thm:weak suobject tripos} and by \Cref{def:exists-sheaf tripos}.

   $\it{(2\Rightarrow 3)}$ If $\Psi_{\mG_P}$ is a full tripos, then by \Cref{thm: full existential doctrine as weak subobject} we have that $\compexfull{P}$ is a first-order hyperdoctrine. Combining \Cref{thm:characterisation tripostotopos full ex comp} with \Cref{cor: Menni chararcterisation} we have that the category $\sT_{{\compexfull{P}}}\equiv \exlex{\mG_P}$ is a topos. Now we can conclude that $\compexfull{P}$ is a full tripos applying \Cref{cor:P_ex_is_tripos_iff_its_ttt_is_ topos}. 
   
   $\it{(3\Rightarrow 4)}$ By \Cref{thm:characterisation tripostotopos full ex comp} we have that $ \sT_{\compexfull{P}}\equiv \exlex{\mG_P}\equiv \presh{P}$, and hence, by \Cref{thm:hyperdoctrine and CA iff T topos}, we can conclude that $\presh{P}$ is a topos, since $\compexfull{P}$ is a full tripos.

   $\it{(4\Rightarrow 1)}$ It follows by \Cref{def:P-presheaves} and by \Cref{cor: Menni chararcterisation}.
   \end{proof}
\begin{cor}\label{remark: fondamentale}
    If a lex primary doctrine $\doctrine{\mC}{P}$ happens to be $\exists$-supercompactifiable, then $\doctrine{\mC}{\Psi_{\mC}}$ is a tripos, and hence $\exlex{\mC}$ is a topos.
\end{cor}
\begin{proof}
    This happens because the weak dependent products and the generic proof of $\mG_P$ induce weak dependent products and a generic proof for $\mC$.
\end{proof}
Combining \Cref{remark: fondamentale} with \Cref{thm:comp ex tripos is equivalent to}, \Cref{prop:useful_adjunction} and \Cref{thm:characterisation tripostotopos full ex comp} we obtain the following:
\begin{cor}
Let $\doctrine{\mC}{P}$ be an $\exists$-supercompactifiable doctrine. Then we have a geometric inclusion of triposes $\Psi_{\mC}\to \compexfull{P}$. Moreover, the right adjoint is a full existential morphism.
\end{cor} 

\begin{cor}\label{cor:geometric embedding toposes exlex}
Let $\doctrine{\mC}{P}$ be an $\exists$-supercompactifiable doctrine. Then we have a geometric embedding of toposes $\exlex{\mC}\to \presh{P}$ whose right adjoint is a morphism of exact categories.
\end{cor}

\begin{example}
By \Cref{ex: real tripos is an existential comp}, we have that the lex primary doctrine $\doctrine{\set}{\pca{A}^{(-)}}$  defined in \Cref{ex: primary doctrine pca} is $\exists$-supercompactifiable, since  its full existential completion is the realizability tripos $\doctrine{\set}{\mathcal{P}}$. In this case, the geometric embedding of \Cref{cor:geometric embedding toposes exlex} is exactly the well-known  geometric inclusion $\set \to \mathsf{RT}(\pca{A})$.
\end{example}

  \begin{example}
Let $\supercomp$ be an inf-semilattice. By \Cref{ex:supercoherent localic doctrine is full ex comp}, we have that the lex primary doctrine $\doctrine{\set}{\supercomp^{(-)}}$ is $\exists$-supercompactifiable, because its full existential completion is the localic tripos $\doctrine{\set}{D(\supercomp)^{(-)}}$.  In this case, the geometric inclusion of \Cref{cor:geometric embedding toposes exlex} is exactly the geometric inclusion $\set \to \sh{D(\supercomp)}$.
	
  \end{example}

\begin{example} By \Cref{cor: Menni chararcterisation} and \Cref{thm:weak suobject tripos}	we have that every trivial doctrine $\doctrine{\mC}{\top}$ is an $\exists$-supercompactifiable tripos if and only if $\mC$ has weak dependent products and a generic proof.
\end{example}
\begin{example}
    A non-trivial example of a lex primary doctrine that is non $\set$-based is the elementary instance reducibility doctrine presented in \cite[Def 4.1]{maschitrottao2025}, since its full existential completion is a tripos \cite[Thm. 4.9]{maschitrottao2025}. This is a doctrine whose base category is that of partitioned assemblies, and whose full existential completion abstracts the notion of extended Weihrauch degrees.  In this case, the geometric embedding of \Cref{cor:geometric embedding toposes exlex} provides exactly the geometric inclusion of the effective topos into the topos of extended Weihrauch degrees proved in \cite[Thm. 6.8]{maschitrottao2025}.
\end{example}

It is important to mention that within the literature of categorical realizability, there are already two results which can be seen as instances of the general theorem   \Cref{thm:comp ex tripos is equivalent to}, one provided by P. Hofstra in \cite{Hofstra2006} and the other shown by J. Frey in \cite{frey2024}. 
These are two combinatorial characterizations of $\set$-indexed preorders associated with suitable pre-realizability notions whose full existential completion is a tripos. 
While in both these works the main focus is on the features of 
pre-realizability notion which is taken as starting point for defining doctrines and triposes (basic combinatory objects and  filtered ordered combinatory algebras in \cite{Hofstra2006}, uniform preorder and discrete combinatory object in \cite{frey2024}), our characterization is purely categorical and designed for arbitrary lex primary doctrines.


\paragraph{Graphical summary} The constructions considered in this section fit into the following commutative diagram:
\[\begin{tikzcd}
	\exists\mbox{-}\fullprDoc && \mathsf{Topoi} \\
	& \mathsf{Trip}
	\arrow["\presh{(-)}",squiggly, from=1-1, to=1-3]
	\arrow["\compexfull{(-)}"',squiggly, from=1-1, to=2-2]
	\arrow["\sT_{(-)}"',squiggly, from=2-2, to=1-3]
\end{tikzcd}\]
where $\mathsf{Trip}$ is the 2-full subcategory of $\EED$ whose objects are full triposes, and $\mathsf{Topoi}$ is the 2-category of toposes whose 1-cells are regular functors.
In particular, we may consider the 2-full subcategory $\exists\mbox{-}\mathsf{Trip}$ of $\mathsf{Trip}$ whose objects are $\exists$-supercompactifiable triposes, obtaining:
\[\begin{tikzcd}
	\exists\mbox{-} \mathsf{Trip} && \mathsf{Topoi} \\
	& \mathsf{Trip}
	\arrow["\presh{(-)}",squiggly, from=1-1, to=1-3]
	\arrow["\compexfull{(-)}"',squiggly, from=1-1, to=2-2]
	\arrow["\sT_{(-)}"',squiggly, from=2-2, to=1-3]
\end{tikzcd}\]

\section{A characterization of $\exists$-supercompactifiable triposes}\label{Sec char ex supercomp triposes}

The main purpose of this section is to provide a simple characterization of $\exists$-supercompactifiable triposes, namely 
\Cref{thm:main} where we prove that a tripos is $\exists$-supercompactifiable if and only if its base category has weak dependent products and a generic proof.

This result allows us to easily show that the full existential completion of an $\exists$-supercompactifiable doctrine is again $\exists$-supercompactifiable, and to find several examples (and also non-examples) of $\exists$-supercompactifiable triposes.

E.g. we can deduce that {\it  every $\set$-based full tripos is an $\exists$-supercompactifiable tripos} by a crucial use of the axiom of choice in the meta-theory.

We first study when a category of elements has weak dependent products. For the sake of generality, we will present all the results with the minimal assumption on the doctrines.
\begin{lemma}\label{prop: properties of weak dep. prod}
Let $\doctrine{\mC}{P}$ be an implicational and universal doctrine whose base category has finite limits.
If the diagram

\[\begin{tikzcd}
	(X,\alpha)  & (E,P_{g^*(h')}(\beta)\wedge P_{(h')^*g}(\sigma'))  & (Z',\sigma') \\
	& (J,\beta)  & (I,\gamma)
	\arrow["e'"',from=1-2, to=1-1]
	\arrow[from=1-2, to=1-3]
	\arrow[from=1-2, to=2-2]
	\arrow["g"',from=2-2, to=2-3]
	\arrow["h' ",from=1-3, to=2-3]
	\arrow["f"',from=1-1, to=2-2]
	\arrow["\scalebox{1.6}{$\lrcorner$}"{anchor=center, pos=0.1}, shift left=3, draw=none, from=1-2, to=2-2]
\end{tikzcd}\]
commutes in $\mG_P$, then we have that
\[\sigma'\leq \forall_{(h')^*g}(P_{g^*(h')}(\beta)\rightarrow P_{e'}(\alpha))\wedge P_{h'}(\gamma).\]
\end{lemma}
\begin{proof}
Since $\freccia{(Z',\sigma')}{h'}{(I,\gamma)}$ is an arrow of $\mG_P$ we have that $\sigma'\leq P_{h'}(\gamma)$. Similarly,  since $\freccia{(E,P_{g^*(h')}(\beta)\wedge P_{(h')^*g}(\sigma'))}{e'}{(X,\alpha)}$ is an arrow of $\mG_P$ we have that $P_{g^*(h')}(\beta)\wedge P_{(h')^*g}(\sigma')\leq P_{e'}(\alpha)$, and hence $\sigma'\leq \forall_{(h')^*g}(P_{g^*(h')}(\beta)\rightarrow P_{e'}(\alpha))$, because $P$ is universal and implicational. Therefore we can conclude that $\sigma'\leq \forall_{(h')^*g}(P_{g^*(h')}(\beta)\rightarrow P_{e'}(\alpha))\wedge P_{h'}(\gamma).$
\end{proof}
\begin{theorem}\label{thm: weak dep. product of G_P}
Let $\doctrine{\mC}{P}$ be a lex primary doctrine whose base category has weak dependent products. If $P$ is implicational and universal, then $\mG_P$ has weak dependent products. In particular, if

\[\begin{tikzcd}
	X & E & Z \\
	& J & I
	\arrow["e"', from=1-2, to=1-1]
	\arrow["f"', from=1-1, to=2-2]
	\arrow["g"', from=2-2, to=2-3]
	\arrow["h", from=1-3, to=2-3]
	\arrow[from=1-2, to=2-2]
	\arrow[from=1-2, to=1-3]
	\arrow["\scalebox{1.6}{$\lrcorner$}"{anchor=center, pos=0.1}, shift left=3, draw=none, from=1-2, to=2-2]
\end{tikzcd}\]
is a weak dependent product of $f$ along $g$ in $\mC$, then the diagram

\[
\begin{tikzcd}
	(X,\alpha)  & (E,P_{g^*h}(\beta)\wedge P_{h^*g}(\sigma))  & (Z,\sigma) \\
	& (J,\beta)  & (I,\gamma)
	\arrow["e"',from=1-2, to=1-1]
	\arrow[from=1-2, to=1-3]
	\arrow[from=1-2, to=2-2]
	\arrow["g"',from=2-2, to=2-3]
	\arrow["h ",from=1-3, to=2-3]
	\arrow["f"',from=1-1, to=2-2]
	\arrow["\scalebox{1.6}{$\lrcorner$}"{anchor=center, pos=0.1}, shift left=3, draw=none, from=1-2, to=2-2]
\end{tikzcd}
\]
with 
\[\sigma:= \forall_{h^*g}(P_{g^*h}(\beta)\rightarrow P_e(\alpha))\wedge P_h(\gamma)\]
is a weak dependent product of $\freccia{(X,\alpha)}{f}{(J,\beta)}$ along  $\freccia{(J,\beta)}{g}{(I,\gamma)}$ in
$\mG_P$.
\end{theorem}

\begin{proof}
First, notice that the diagram is well-defined in $\mG_P$ because, by definition, we have that $\sigma \leq P_h(\gamma)$ and $P_{g^*h}(\beta)\wedge P_{h^*g}(\sigma)\leq P_e(\alpha)$.
Now we show that $\freccia{(Z,\sigma)}{h}{(I,\gamma)}$ satisfies the universal property of weak dependent products. Let us consider another diagram
\[\begin{tikzcd}
	(X,\alpha)  & (E',P_{g^*(h')}(\beta)\wedge P_{(h')^*g}(\sigma'))  & (Z',\sigma') \\
	& (J,\beta)  & (I,\gamma).
	\arrow["e'"',from=1-2, to=1-1]
	\arrow[from=1-2, to=1-3]
	\arrow[from=1-2, to=2-2]
	\arrow["g"',from=2-2, to=2-3]
	\arrow["h' ",from=1-3, to=2-3]
	\arrow["f"',from=1-1, to=2-2]
	\arrow["\scalebox{1.6}{$\lrcorner$}"{anchor=center, pos=0.1}, shift left=3, draw=none, from=1-2, to=2-2]
\end{tikzcd}\]
Since $\freccia{Z}{h}{I}$ is a weak dependent product of $f$ along $g$ in $\mC$, there exists an arrow $\freccia{Z'}{w}{Z}$ (and an arrow $\freccia{E'}{k}{E}$) such that the diagram 

\[\begin{tikzcd}
	&& {E'} & {Z'} \\
	X & E & Z \\
	& J & I
	\arrow["e"', from=2-2, to=2-1]
	\arrow["f"', from=2-1, to=3-2]
	\arrow["g"', from=3-2, to=3-3]
	\arrow[from=2-2, to=3-2]
	\arrow["k", dashed, from=1-3, to=2-2]
	\arrow[from=1-3, to=1-4]
	\arrow["h", from=2-3, to=3-3]
	\arrow[curve={height=-18pt}, from=1-3, to=3-2]
	\arrow["{h'}", curve={height=-18pt}, from=1-4, to=3-3]
	\arrow["w", dashed, from=1-4, to=2-3]
	\arrow["{e'}"', curve={height=12pt}, from=1-3, to=2-1]
	\arrow["\scalebox{1.6}{$\lrcorner$}"{anchor=center, pos=0.1}, shift left=3, draw=none, from=2-2, to=3-2]
	\arrow[from=2-2, to=2-3, crossing over]
\end{tikzcd}\]
commutes.  We have to show that the arrow $\freccia{Z'}{w}{Z}$ induces an arrow $\freccia{(Z',\sigma')}{w}{(Z,\sigma)}$ in $\mG_P$, i.e. that $\sigma'\leq P_w(\sigma)$. By \Cref{prop: properties of weak dep. prod} we have that
\[\sigma'\leq \forall_{(h')^*g}(P_{g^*(h')}(\beta)\rightarrow P_{e'}(\alpha))\wedge P_{h'}(\gamma).\]
Hence, since $h'=hw$ and $e'=ek$, we have that
\[\sigma'\leq \forall_{(hw)^*g}(P_{g^*(hw)}(\beta)\rightarrow P_{ek}(\alpha))\wedge P_{hw}(\gamma).\]
Therefore, since we have that $g^*(hw)=(g^*(h))k$, we have that 
\[ \sigma'\leq \forall_{(hw)^*g}P_k(P_{g^*h}(\beta)\rightarrow P_{e}(\alpha))\wedge P_{hw}(\gamma).\]
Now, by Beck-Chevalley condition (and the standard property of pullbacks), we have that $\forall_{(hw)^*(g)}P_k= P_w \forall_{h^*(g)}$, and hence we can conclude 
\[ \sigma'\leq P_w( \forall_{h^*(g)}(P_{g^*(h)}(\beta)\rightarrow P_{e}(\alpha))\wedge P_{h}(\gamma))=P_w(\sigma).\]
Finally, it is straightforward that the arrow $k$ is well-defined in $\mG_P$.
\end{proof}

\begin{example}
The category of elements of any $\set$-based tripos has weak dependent products.
\end{example}
\begin{example}
	The category of elements of the doctrine $\doctrine{\mC}{\Psi_{\mC}}$ has weak dependent products whenever $\mC$ has them.
\end{example}

Now we are going to study the generic proof in the category of elements:
\begin{theorem}\label{thm: G_P has generic proof}
    Let $\doctrine{\mC}{P}$ be a lex primary doctrine. If $P$ has a weak predicate classifier and $\mC$ has weak dependent products and a generic proof, then the category $\mG_P$ has a generic proof.
\end{theorem}
\begin{proof}
   Let us consider an arrow $\freccia{(Y,\alpha)}{f}{(X,\beta)}$ of $\mG_P$. By definition of weak predicate classifier of $P$, we have an arrow $\freccia{Y}{\{\alpha\}}{\Omega}$ such that $\alpha= P_{\{\alpha\}}(\in)$. Now, by \Cref{thm:weak suobject tripos} we have that $\doctrine{\mC}{\Psi_{\mC}}$ is a full tripos, and hence we have weak power objects. Hence we can consider the arrow  $\freccia{Y}{\angbr{\{\alpha\}}{f}}{\Omega \times X}$ and, by definition of power objects for $\Psi_{\mC}$, we have a commutative diagram
\begin{equation}\label{diag:importante in C}
\begin{tikzcd}[column sep=huge, row sep=huge ]
Y 
  & E
  & \Theta_{\Omega} \\
& \Omega \times X 
  & \Omega \times \mathrm{P}(\Omega) \\
& X 
  & \mathrm{P} (\Omega)
  \arrow["\angbr{\{\alpha\}}{f}"', from=1-1, to=2-2]
\arrow["e_2"',shift right=.5ex, from=1-1, to=1-2]
\arrow["e_1"', shift right=.5ex, from=1-2, to=1-1]
\arrow["a_2", from=1-2, to=1-3]
\arrow["a_1"', from=1-2, to=2-2]
\arrow["\in_{\Omega}", from=1-3, to=2-3]
\arrow["{\mathrm{id}_{\Omega} \times \{\angbr{\{\alpha\}}{f}\}_{\Omega}}"', from=2-2, to=2-3]
\arrow["\pr_X"', from=2-2, to=3-2]
\arrow["\pr_{ \mathrm{P} \Omega}", from=2-3, to=3-3]
\arrow[" \{\angbr{\{\alpha\}}{f}\}_{\Omega}"', from=3-2, to=3-3]
\arrow["f"', bend right=35, from=1-1, to=3-2]
\arrow["\scalebox{1.6}{$\lrcorner$}"{anchor=center, pos=0.15}, shift left=3, draw=none, from=1-2, to=2-2]
\arrow["\scalebox{1.6}{$\lrcorner$}"{anchor=center, pos=0.15}, shift left=3, draw=none, from=2-2, to=3-2]
\end{tikzcd}
\end{equation}
in $\mC$ where the arrows $\freccia{\Theta_{\Omega}}{\in_\Omega}{ \Omega \times \mathrm{P}(\Omega)}$ and $\freccia{X}{ \{\angbr{\{\alpha\}}{f}\}_{\Omega}}{\mathrm{P} (\Omega)}$ are defined by the power objects in $\Psi_{\mC}$.
Now we show that the diagram:
   \[
\begin{tikzcd}[column sep=large, row sep=huge ]
(Y,\alpha) 
  & (E,P_{\pr_Xa_1}(\beta)\wedge P_{\pr_{\Omega}\in_{\Omega}a_2}(\in))
  & (\Theta_{\Omega},P_{\pr_{\Omega}\in_{\Omega}}(\in)) \\
& (\Omega \times X, P_{\pr_X}(\beta))
  & (\Omega \times \mathrm{P}(\Omega),\top_{\Omega \times \mathrm{P}(\Omega)}) \\
& (X,\beta)
  & (\mathrm{P} (\Omega),\top_{\mathrm{P} (\Omega)})
  \arrow["\angbr{\{\alpha\}}{f}"', from=1-1, to=2-2]
\arrow["e_2"',shift right=.5ex, from=1-1, to=1-2]
\arrow["e_1"', shift right=.5ex, from=1-2, to=1-1]
\arrow["a_2", from=1-2, to=1-3]
\arrow["a_1"', from=1-2, to=2-2]
\arrow["\in_{\Omega}", from=1-3, to=2-3]
\arrow["{\mathrm{id}_{\Omega} \times \{\angbr{\{\alpha\}}{f}\}_{\Omega}}"', from=2-2, to=2-3]
\arrow["\pr_X"', from=2-2, to=3-2]
\arrow["\pr_{ \mathrm{P} (\Omega)}", from=2-3, to=3-3]
\arrow[" \{\angbr{\{\alpha\}}{f}\}_{\Omega}"', from=3-2, to=3-3]
\arrow["f"', bend right=35, from=1-1, to=3-2]
\arrow["\scalebox{1.6}{$\lrcorner$}"{anchor=center, pos=0.15}, shift left=3, draw=none, from=1-2, to=2-2]
\arrow["\scalebox{1.6}{$\lrcorner$}"{anchor=center, pos=0.15}, shift left=3, draw=none, from=2-2, to=3-2]
\end{tikzcd}
\]
is a well-defined diagram in $\mG_P$, namely that the arrow  
$$\freccia{(\Theta_{\Omega},P_{\pr_{\Omega}\in_{\Omega}}(\in))}{\pr_{ \mathrm{P} \Omega}\in_{\Omega}}{(\mathrm{P} (\Omega),\top_{\mathrm{P} (\Omega)})}$$
 is a generic proof for $\mG_P$. Notice that we need only to check that $e_1$ and $e_2$ are well-defined arrows of $\mG_P$ because both the pullbacks are well-defined in $\mG_P$ just by definition. 

Now we start by showing that 
 \[P_{\pr_Xa_1}(\beta)\wedge P_{\pr_{\Omega}\in_{\Omega}a_2}(\in)= P_{e_1}(\alpha).\] 
This follows since \eqref{diag:importante in C} is a commutative diagram in $\mC$, and in particular we have that $a_1=\angbr{\{\alpha\}}{f}e_1$ and $\pr_{\Omega}\in_{\Omega}a_2= \pr_{\Omega}( \mathrm{id}_{\Omega} \times \{\angbr{\{\alpha\}}{f}\}_{\Omega})a_1$, hence
\[\pr_{\Omega}\in_{\Omega}a_2=\pr_{\Omega}( \mathrm{id}_{\Omega} \times \{\angbr{\{\alpha\}}{f}\}_{\Omega})\angbr{\{\alpha\}}{f}e_1= \{\alpha\}e_1.\]
In particular, using also the fact that $\alpha= P_{\{\alpha\}}(\in)$, we have that
\[P_{\pr_Xa_1}(\beta)\wedge P_{\pr_{\Omega}\in_{\Omega}a_2}(\in)=P_{\pr_X\angbr{\{\alpha\}}{f}e_1}(\beta)\wedge P_{\{\alpha\}e_1}(\in)=P_{e_1}(P_{f}(\beta)\wedge \alpha)=P_{e_1}(\alpha)\]
where the last equality follows from the assumption that $\freccia{(Y,\alpha)}{f}{(X,\beta)}$ is an arrow of $\mG_P$, i.e. $\alpha\leq P_f(\beta)$.
Hence, we have that $\freccia{(E,P_{e_1}(\alpha))}{e_1}{(Y,\alpha)}$ is a well-defined arrow of $\mG_P$. 

We next show that 
\[\alpha= P_{e_1e_2}(\alpha).\]
By definition of weak predicate classifier of $P$, we have that $\alpha=P_{\{\alpha\}}(\in)$. Now, since \eqref{diag:importante in C} is a commutative diagram, 
we have that $\alpha=P_{\pr_{\Omega} a_1e_2}(\in)$
and $\angbr{\{\alpha\}}{f}=a_1e_2$ and $a_1=\angbr{\{\alpha\}}{f}e_1$. Therefore we can conclude that
\[ \alpha = P_{\pr_{\Omega} a_1e_2}(\in)=   P_{\pr_{\Omega} \angbr{\{\alpha\}}{f}e_1 e_2}(\in)= P_{e_1e_2}(P_{\{\alpha\}}(\in))=P_{e_1e_2}(\alpha)= P_{e_2}(P_{e_1}(\alpha)) \]
and hence that  $\freccia{(Y,\alpha)}{e_2}{(E,P_{e_1}(\alpha))}$ is a well-defined  arrow of $\mG_P$.
\end{proof}

 Now, we are ready to obtain our main result:

\begin{cor}\label{thm:main}
    Let $\doctrine{\mC}{P}$ be a full tripos. Then the following conditions are equivalent:
    \begin{enumerate}
        \item $P$ is $\exists$-supercompactifiable;
        \item $\mC$ has weak dependent products and a generic proof;
        \item $\doctrine{\mC}{\Psi_{\mC}}$ is a full tripos.
    \end{enumerate}
\end{cor}
\begin{proof}
$1)\Rightarrow 2)$ Follows by \Cref{remark: fondamentale} and \Cref{thm:weak suobject tripos}. $2)\Rightarrow 1) $ follows by \Cref{thm: G_P has generic proof} and \Cref{thm: weak dep. product of G_P}.
 $2) \iff 3)$ follows by \Cref{thm:weak suobject tripos}.
\end{proof}
Combining this corollary with \Cref{remark: fondamentale} we obtain the following result showing that the property of being $\exists$-supercompactifiable is preserved by the full existential completion:
\begin{cor}\label{cor: stability of ex comp for supercompact doc}
Let $\doctrine{\mC}{P}$ be an $\exists$-supercompactifiable doctrine. Then $\compexfull{P}$ is an $\exists$-supercompactifiable tripos.
\end{cor}
\begin{cor}\label{cor_every_topos_with_splity_epi_tripos_is_sheaf_tripos}
    Every full tripos whose base category is a topos with splitting epis is $\exists$-supercompactifiable.
\end{cor}
In particular, we can conclude that every tripos as originally introduced in \cite{TT} is $\exists$-supercompactifiable:
\begin{cor}\label{cor_every_set_base_tripos_is_sheaf_tripos}
Every full tripos whose base category is $\set$ (with the axiom of choice) is $\exists$-supercompactifiable. 
\end{cor}
\begin{example}
    Relevant examples of $\set$-based triposes, which are in particular $\exists$-supercompactifiable by \Cref{cor_every_set_base_tripos_is_sheaf_tripos}, are localic and realizability triposes \cite{TT,TTT}, the modified realizability tripos \cite{Hylandmodifiedreal,VANOOSTENmodifiedreal}, the extensional realizability tripos \cite{VANOOSTENereal}, the dialectica tripos \cite{Biering2008}, Krivine tripos \cite{STREICHER_2013}.
    \end{example}
    \begin{example}
        For every topos $\mathcal{B}$ with splitting epis, we have that the full tripos $\doctrine{\mathcal{B}}{\Sub_{\mathcal{B}}}$ is $\exists$-supercompactifiable. Relevant examples of this class of toposes are presheaf toposes over a discrete category.
    \end{example}

\begin{example}\label{ex_weih_tripos_is_supercomp}
    The extended Weihrauch tripos $\doctrine{\parassembly{\pca{A}}}{\mathfrak{eW}}$  introduced in \cite{maschitrottao2025} is $\exists$-supercompactifiable, since the category of partitioned assemblies $\parassembly{\pca{A}}$ has weak dependent products and a generic proof.
\end{example}

\begin{example}[non-$\exists$-supercompactifiable triposes]
    Every subobject doctrine   $\doctrine{\mathcal{B}}{\Sub_{\mathcal{B}}}$ on an elementary topos $\mathcal{B}$ with no generic proof provides an example of a tripos that is not $\exists$-supercompactifiable. See \cite[Sec. 5.3]{MENNI2003287} for examples of  toposes with no generic proofs.
    
    Trivial non-examples of $\exists$-supercompactifiable triposes are provided by the constant top tripos  $\doctrine{\mC}{\top}$ on any finite limit category $\mC$ with weak dependent products and no generic proof. Examples of such categories include the syntactic categories of the internal type theory of locally cartesian closed categories in \cite{modular}, whose exact on lex completion cannot be a topos for proof-theoretic reasons.
\end{example}  
\paragraph{Graphical summary}
The main results of this section, namely \Cref{thm:main} and \Cref{cor: stability of ex comp for supercompact doc}, can be summarized by the following diagrams:
\[\begin{tikzcd}
	\exists\mbox{-}\fullprDoc && \mathsf{Topoi} \\
	& \exists\mbox{-}\mathsf{Trip}
	\arrow["\presh{(-)}",squiggly, from=1-1, to=1-3]
	\arrow["\compexfull{(-)}"',squiggly, from=1-1, to=2-2]
	\arrow["\sT_{(-)}"',squiggly, from=2-2, to=1-3]
\end{tikzcd}\]
and hence
\[\begin{tikzcd}
	\exists\mbox{-}\mathsf{Trip} && \mathsf{Topoi} \\
	& \exists\mbox{-}\mathsf{Trip}
	\arrow["\presh{(-)}",squiggly, from=1-1, to=1-3]
	\arrow["\compexfull{(-)}"',squiggly, from=1-1, to=2-2]
	\arrow["\sT_{(-)}"',squiggly, from=2-2, to=1-3]
\end{tikzcd}\]
Moreover, since localic triposes are $\exists$-supercompactifiable, we have that the previous diagram really generalizes the localic case, providing the following commutative diagram summarizing our results

\[\begin{tikzcd}
	&  \mathsf{Topoi} &&& \mathsf{Topoi}  \\
	\exists\mbox{-}\mathsf{Trip}  &&& \mathsf{Loc} \\
	& \exists\mbox{-}\mathsf{Trip} &&& \mathsf{Loc}
	\arrow[squiggly, hook', from=1-5, to=1-2]
	\arrow["\presh{(-)}",squiggly, from=2-1, to=1-2]
	\arrow["\compexfull{(-)}"',squiggly, from=2-1, to=3-2]
	\arrow["\mathsf{PSh}(-)",squiggly, from=2-4, to=1-5]
	\arrow[squiggly, hook', from=2-4, to=2-1]
	\arrow["D(-)"', squiggly, from=2-4, to=3-5]
	\arrow["\sT_{(-)}"'{pos=0.6},squiggly, from=3-2, to=1-2]
	\arrow["\mathsf{Sh}(-)"'{pos=0.6},squiggly, from=3-5, to=1-5]
	\arrow[squiggly, hook', from=3-5, to=3-2]
\end{tikzcd}\]
where the assignment from $\mathsf{Loc}$ into $\exists\mbox{-}\mathsf{Trip}$ is given by $\loc \mapsto \loc^{(-)}$ and the assignment of $\mathsf{Topoi}$ into itself is the identity.
\subsection{Closure of $\exists$-supercompactifiable triposes under slicing}
The main purpose of this section is to prove a fundamental-like theorem for $\exists$-supercompactifiable triposes. To achieve this goal, we start by recalling from \cite[Ex. 3.15]{CIOFFO2025} the notion of \emph{slice doctrine}:
\begin{definition}[slice doctrine]\label{ex:slice doctrine}
    Let $\doctrine{\mC}{P}$ be a lex primary doctrine and $X$ be an object of $\mC$. The \textbf{slice doctrine} $\doctrine{\mC/X}{P_{/X}}$ is the functor defined by the assignments $P_{/X}(\frecciasopra{Y}{f}{X}):=P(Y)$ and, for an arrow $\freccia{g}{h}{f}$ in $\mC/X$, $P_{/X}(h):=P_h$.
\end{definition}
\begin{proposition}\label{prop:fundamental thm triposes}
    Let $\doctrine{\mC}{P}$ be a lex primary doctrine and $X$ be an object of $\mC$. Then:
    \begin{enumerate}
        \item if $\mC$ has weak dependent products and a generic proof, $\mC/X$ has weak dependent products and a generic proof;
        \item if $P$ is a  full first-order hyperdoctrine, $P_{/X}$ is a full first-order hyperdoctrine;
        \item if $P$ has a weak predicate classifier, $P_{/X}$ has a weak predicate classifier.
    \end{enumerate}
\end{proposition}
\begin{proof}
(1) Recall from \cite{EMMENEGGER} that if a category has finite limits and weak dependent products then every slice category has weak dependent products. Moreover, if $\mC$ has a generic proof  $\freccia{\Theta}{\theta}{\Lambda}$, it is straightforward to check that every slice category $\mC/X$ has a generic proof  given by 
\[\begin{tikzcd}
	{X\times \Theta} && {X\times \Lambda} \\
	& X
	\arrow["{\id_X\times \theta}", from=1-1, to=1-3]
	\arrow["{\pi_X}"', from=1-1, to=2-2]
	\arrow["{\pi_X}", from=1-3, to=2-2]
\end{tikzcd}\]

\noindent
(2) It follows by definition of slice doctrine.

\noindent
(3) If $\doctrine{\mC}{P}$ is a lex primary doctrine with a weak predicate classifier given by an object $\Omega$ and $\in$ element of $P(\Omega)$ then the slice doctrine $\doctrine{\mC/X}{P_{/X}}$ as defined in \Cref{ex:slice doctrine} has a weak predicate classifier, given by the object $\freccia{\Omega\times X}{\pr_X}{X}$ and the element $P_{\pr_{\Omega}}(\in)$.
\end{proof}
Combining \Cref{prop:fundamental thm triposes} with \Cref{thm:main} we obtain the following result:

\begin{cor}\label{cor localization}
    If a full tripos $\doctrine{\mC}{P}$ is $\exists$-supercompactifiable, then $\doctrine{\mC/X}{P_{/X}}$ is an $\exists$-supercompactifiable full tripos for every object $X$ of $\mC$.
 \end{cor}

\section{Abstracting localic sheafification}\label{sec abstract sheafification}
Our previous analysis enables us to show that for an $\exists$-supercompactifiable tripos, its tripos-to-topos construction is equivalent to a category of $j$-sheaves for a Lawvere-Tierney topology that abstracts the sheafification process. In particular, we prove that this topology arises from a geometric embedding of triposes:

\begin{theorem}\label{thm:main 0}
Let $\doctrine{\mC}{P}$ be an $\exists$-supercompactifiable tripos. Then
\begin{enumerate}

\item there is a geometric embedding of triposes $P\hookrightarrow \compexfull{P}$;

\item there is a geometric embedding of toposes $\sT_P\hookrightarrow \sT_{\compexfull{P}}$.
\
\end{enumerate}
\end{theorem}
\begin{proof}
The first point follows from \Cref{cor:adjunction} and the second by \Cref{cor geometric morphism of triposes induce geometric morphisms}.
\end{proof}
Since geometric embeddings of toposes correspond to  Lawvere-Tierney topologies, see e.g.  \cite[Cor.~7, Sec. VII]{SGL}, we have the following corollary:
\begin{cor}\label{cor_sheafification_Lw_Tr}
Let $\doctrine{\mC}{P}$ be an $\exists$-supercompactifiable tripos. Then there exists a Lawvere-Tierney topology $j$ on $\sT_{\compexfull{P}}$ such that $\sT_{P}\equiv \shtopology{\sT_{\compexfull{P}}}{j}$.
\end{cor}
 Combining \Cref{ex: exact comp of weak sub 1} and \Cref{thm:weak suobject tripos} with \Cref{cor_sheafification_Lw_Tr} we have that every topos $\mathcal{B}$ arising as the $\mathsf{ex/lex}$-completion of a lex category $\mC$ fits within our framework:
\begin{cor}
   Let $\mathcal{B}\equiv \exlex{\mC}$ be a topos. Then $\mathcal{B}\equiv \shtopology{\sT_{\compexfull{\Psi_{\mC}}}}{j}$.
\end{cor}

Combining \Cref{cor_sheafification_Lw_Tr} with \Cref{cor_every_topos_with_splity_epi_tripos_is_sheaf_tripos} we obtain:
\begin{cor}
Let $\doctrine{\mathcal{B}}{P}$ be a full tripos, where $\mathcal{B}$ is a topos whose epis split. Then there exists a Lawvere-Tierney topology $j$ on $\sT_{\compexfull{P}}$ such that $\sT_{P}\equiv \shtopology{\sT_{\compexfull{P}}}{j}$.
\end{cor}
In particular, by \Cref{cor_every_set_base_tripos_is_sheaf_tripos}, we have:

\begin{cor}\label{cor_sheafification_Lw_Tr_set}
Let $\doctrine{\set}{P}$ be a full tripos. Then there exists a Lawvere-Tierney topology $j$ on $\sT_{\compexfull{P}}$ such that $\sT_{P}\equiv \shtopology{\sT_{\compexfull{P}}}{j}$.
\end{cor}

\begin{example}[Localic toposes]
Let $\doctrine{\set}{\loc^{(-)}}$ be a localic tripos, as defined in \Cref{ex:localic tripos}. By \Cref{thm:main}, the doctrine $\doctrine{\set}{\compexfull{\loc}}$ is a full tripos and, by \Cref{thm:main 0}, we have an adjunction between the toposes 
\[\begin{tikzcd}[column sep=normal, row sep=normal]
	\mathsf{PSh}(\loc) && \sh{\loc}
	\arrow[hook',""{name=0, anchor=center, inner sep=0}, "{\sT(\id_{\set},\mathfrak{i})}",curve={height=-18pt}, from=1-3, to=1-1]
	\arrow[""{name=1, anchor=center, inner sep=0},"{\sT(\id_{\set},\bar{\mathfrak{i}})}", curve={height=-18pt}, from=1-1, to=1-3]
	\arrow["\scalebox{1.3}{$\dashv$}"{anchor=center, rotate=-90}, draw=none, from=1, to=0]
\end{tikzcd}\]
that is the localic sheafification.
\end{example}

\begin{example}[Realizability toposes]
Let $\pca{A}$ be a pca, and let us consider the realizability tripos  $\doctrine{\set}{\mathcal{P}}$. Then the category $ \presh{\mathcal{P}}$ is a topos, and the inclusion $\mathsf{RT}(\pca{A}) \hookrightarrow \presh{\mathcal{P}}$ has a left adjoint. In particular, by \Cref{thm:main 0}, we have an adjunction
\[\begin{tikzcd}[column sep=normal, row sep=normal]
	 \presh{\mathcal{P}} && \mathsf{RT}(\pca{A}).
	\arrow[hook',""{name=0, anchor=center, inner sep=0}, "{\sT(\id_{\set},\mathfrak{i})}",curve={height=-18pt}, from=1-3, to=1-1]
	\arrow[""{name=1, anchor=center, inner sep=0},"{\sT(\id_{\set},\bar{\mathfrak{i}})}", curve={height=-18pt}, from=1-1, to=1-3]
	\arrow["\scalebox{1.3}{$\dashv$}"{anchor=center, rotate=-90}, draw=none, from=1, to=0]
\end{tikzcd}\]
Hence, we have that every realizability topos can be presented as a topos of $j$-sheaves $\mathsf{RT}(\pca{A})\equiv \shtopology{ \presh{\mathcal{P}}}{j}$ with respect to a topos and a topology abstracting the topos of localic presheaves and the localic sheafification.
\end{example}

\begin{example}
Since all these toposes are generated via the tripos-to-topos of a $\set$-based tripos, we have that the modified realizability topos \cite{Hylandmodifiedreal,VANOOSTENmodifiedreal}, the extensional realizability topos \cite{VANOOSTENereal}, dialectica topos \cite{Biering2008}, Krivine topos \cite{STREICHER_2013} can all be presented as toposes of $j$-sheaves with respect to the topos of abstract presheaves generated by their generating triposes.
\end{example}
\begin{example}
    By \Cref{ex_weih_tripos_is_supercomp}, we have that the extended Weihrauch topos fits within our framework, hence it can be presented as a topos of $j$-sheaves with respect to the topos of abstract presheaves generated by its generating tripos.
\end{example}

\section{Conclusion}\label{conclusioni}
We have shown that localic and realizability toposes share deeper structural features beyond being instances of the tripos-to-topos construction  \cite{TT,TTT}: they are  both instances of 
a tripos-to-topos $\sT_P$ generated from 
an $\exists$-supercompactifiable tripos $P$. This means that they both can be presented as toposes of $j$-sheaves on an abstract localic
presheaf category $\presh{P}=\exlex{\mG_P}$  which  enjoys a generalization of the localic instance of the Comparison Lemma, since their $\presh{P}$ coincides with the topos $\sT_{\compexfull{P}}$
generated by their full existential completion $\compexfull{P}$. Diagrammatically, we have the geometric embeddings of toposes
\[\begin{tikzcd}
	\mathsf{PSh}(\loc) && \sh{\loc} && \presh{\mathcal{P}} && \mathsf{RT}(\pca{A})
	\arrow[""{name=0, anchor=center, inner sep=0}, curve={height=-18pt}, from=1-1, to=1-3]
	\arrow[""{name=1, anchor=center, inner sep=0}, curve={height=-18pt}, hook', from=1-3, to=1-1]
	\arrow[""{name=2, anchor=center, inner sep=0}, curve={height=-18pt}, from=1-5, to=1-7]
	\arrow[""{name=3, anchor=center, inner sep=0}, curve={height=-18pt}, hook', from=1-7, to=1-5]
	\arrow["\dashv"{anchor=center, rotate=-90}, draw=none, from=0, to=1]
	\arrow["\dashv"{anchor=center, rotate=-90}, draw=none, from=2, to=3]
\end{tikzcd}\]
where $ \mathsf{PSh}(\loc)=\exlex{\mG_{\loc^{(-)}}}\equiv \sT_{\compexfull{\loc}}$ and $ \presh{\mathcal{P}}=\exlex{\mG_{\mathcal{P}}}\equiv \sT_{\compexfull{\mathcal{P}}}$, and more generally 
\[\begin{tikzcd}
	\sT_{\compexfull{P}} && \sT_P
	\arrow[""{name=0, anchor=center, inner sep=0}, curve={height=-18pt}, from=1-1, to=1-3]
	\arrow[""{name=1, anchor=center, inner sep=0}, curve={height=-18pt}, hook', from=1-3, to=1-1]
	\arrow["\dashv"{anchor=center, rotate=-90}, draw=none, from=0, to=1]
\end{tikzcd}\]
where $\presh{P} =\exlex{\mG_P}\equiv \sT_{\compexfull{P}}$.

Regarding related works, we  remark  that in the context of  Grothendieck toposes,
a notion of \emph{supercompact objects} and \emph{supercompactly generated topos} was introduced in \cite{rogers2021,caramello2018}, inspired by \cite{moerdijk2000} but with a crucial use of 
arbitrary coproducts to investigate their properties and applications. In particular, for these notions
we are not aware of  a corresponding
notion of supercompactification,  and, hence, of a generalization of  the localic instance of the 
Comparison Lemma.

In future work, we intend to investigate the 2-categorical aspects of the constructions considered here. In particular, we aim to study their functoriality for suitable choices of morphisms and the coherence of the comparisons established in this paper.
The diagrams presented here describe constructions on objects and commute up to equivalence; a systematic treatment of their 2-categorical structure lies beyond the scope of the present work.

Moreover, we aim to extend our results  to the 
context of fibrations of toposes,  
encompassing predicative localic and realizability toposes such as the one in \cite{MAIETTI_MASCHIO_2021}, where Lawvere doctrines and Lawvere-Tierney sheaves play a crucial role.

\section*{Acknowledgments}
We acknowledge very useful discussions with Martin Hyland, Samuele Maschio, and Giuseppe Rosolini on related topics to this paper. 
The first author is affiliated with the INdAM national research group GNSAGA.

\bibliographystyle{elsarticle-harv}
\bibliography{biblio_davide}

@article{Barr_1986, 
title={Fuzzy Set Theory and Topos Theory}, 
volume={29}, 
DOI={10.4153/CMB-1986-079-9}, 
number={4}, 
journal={Canadian Mathematical Bulletin}, 
author={Barr, M.}, 
year={1986}, 
pages={501-508}
}

@book{SGL,
   AUTHOR = {Mac Lane, Saunders and Moerdijk, Ieke},
     TITLE = {Sheaves in geometry and logic},
    SERIES = {Universitext},
      NOTE = {A first introduction to topos theory,
              Corrected reprint of the 1992 edition},
 PUBLISHER = {Springer-Verlag, New York},
      YEAR = {1994},
     PAGES = {xii+629},
      ISBN = {0-387-97710-4},
   MRCLASS = {03G30 (18B25 54B40)},
  MRNUMBER = {1300636},
MRREVIEWER = {M.\ Makkai},
}

@article{QCFF,
    AUTHOR = {Maietti, Maria Emilia and Rosolini, Giuseppe},
     TITLE = {Quotient completion for the foundation of constructive
              mathematics},
  JOURNAL = {Logica Universalis},
    VOLUME = {7},
      YEAR = {2013},
    NUMBER = {3},
     PAGES = {371--402},
      ISSN = {1661-8297,1661-8300},
   MRCLASS = {03G30 (03B15 03B20 03F55 18C50)},
  MRNUMBER = {3093880},
MRREVIEWER = {Steve\ Awodey},
       DOI = {10.1007/s11787-013-0080-2},
       URL = {https://doi.org/10.1007/s11787-013-0080-2},
}

@article{UEC,
    AUTHOR = {Maietti, Maria Emilia and Rosolini, Giuseppe},
     TITLE = {Unifying exact completions},
  JOURNAL = {Applied Categorical Structures},
    VOLUME = {23},
      YEAR = {2015},
    NUMBER = {1},
     PAGES = {43--52},
      ISSN = {0927-2852,1572-9095},
   MRCLASS = {03G30 (03B15 03B20 03F55 18C50)},
  MRNUMBER = {3297876},
MRREVIEWER = {Robert\ S.\ Lubarsky},
       DOI = {10.1007/s10485-013-9360-5},
       URL = {https://doi.org/10.1007/s10485-013-9360-5},
}

@article{CIOFFO2025,
title = {Biased elementary doctrines and quotient completions},
journal = {Journal of Pure and Applied Algebra},
volume = {229},
number = {6},
pages = {107983},
year = {2025},
issn = {0022-4049},
doi = {https://doi.org/10.1016/j.jpaa.2025.107983},
url = {https://www.sciencedirect.com/science/article/pii/S0022404925001227},
author = {Cipriano Junior Cioffo}
}

@article{EQC,
     AUTHOR = {Maietti, Maria Emilia and Rosolini, Giuseppe},
     TITLE = {Elementary quotient completion},
  JOURNAL = {Theory and Applications of Categories},
    VOLUME = {27},
      YEAR = {2012},
     PAGES = {Paper No. 17, 463},
      ISSN = {1201-561X},
   MRCLASS = {03G30 (03B15 03B20 03F55 18C50)},
  MRNUMBER = {3699058},
MRREVIEWER = {Thomas\ Streicher},
}

@article{SFEC,
     AUTHOR = {Carboni, A.},
     TITLE = {Some free constructions in realizability and proof theory},
  JOURNAL = {Journal of Pure and Applied Algebra},
    VOLUME = {103},
      YEAR = {1995},
    NUMBER = {2},
     PAGES = {117--148},
      ISSN = {0022-4049,1873-1376},
   MRCLASS = {03G30 (03F99 18B25)},
  MRNUMBER = {1358759},
MRREVIEWER = {Colin\ McLarty},
       DOI = {10.1016/0022-4049(94)00103-P},
       URL = {https://doi.org/10.1016/0022-4049(94)00103-P},
}

@article{REC,
   AUTHOR = {Carboni, A. and Vitale, E. M.},
     TITLE = {Regular and exact completions},
  JOURNAL = {Journal of Pure and Applied Algebra},
    VOLUME = {125},
      YEAR = {1998},
    NUMBER = {1-3},
     PAGES = {79--116},
      ISSN = {0022-4049,1873-1376},
   MRCLASS = {18A35 (18A32 18B25 18C20)},
  MRNUMBER = {1600009},
MRREVIEWER = {Walter\ Tholen},
       DOI = {10.1016/S0022-4049(96)00115-6},
       URL = {https://doi.org/10.1016/S0022-4049(96)00115-6},
}

@article{TECH,
      AUTHOR = {Maietti, Maria Emilia and Pasquali, Fabio and Rosolini,
              Giuseppe},
     TITLE = {Triposes, exact completions, and {H}ilbert's
              {$\varepsilon$}-operator},
  JOURNAL = {Tbilisi Mathematical Journal},
    VOLUME = {10},
      YEAR = {2017},
    NUMBER = {3},
     PAGES = {141--166},
      ISSN = {1875-158X,1512-0139},
   MRCLASS = {18B25 (03B15 03G30)},
  MRNUMBER = {3725757},
MRREVIEWER = {Andrzej\ Wi\'snicki},
       DOI = {10.1515/tmj-2017-0106},
       URL = {https://doi.org/10.1515/tmj-2017-0106},
}

@book{CLTT,
    AUTHOR = {Jacobs, Bart},
     TITLE = {Categorical logic and type theory},
    SERIES = {Studies in Logic and the Foundations of Mathematics},
    VOLUME = {141},
 PUBLISHER = {North-Holland Publishing Co., Amsterdam},
      YEAR = {1999},
     PAGES = {xviii+760},
      ISBN = {0-444-50170-3},
   MRCLASS = {03G30 (03-02 03B15 03B40)},
  MRNUMBER = {1674451},
MRREVIEWER = {Andreas\ Blass},
}

@article{AF,
    AUTHOR = {Lawvere, F. William},
     TITLE = {Adjointness in foundations},
      NOTE = {, Reprinted from Dialectica {\bf 23} (1969)},
  JOURNAL = {Reprints in Theory and Applications of Categories},
    NUMBER = {16},
      YEAR = {2006},
     PAGES = {1--16},
   MRCLASS = {18B25 (00A30 18A15 18A40 18B30 55U40)},
  MRNUMBER = {2223032},
MRREVIEWER = {Steve\ Awodey},
}

@article {modular,
    AUTHOR = {Maietti, M. E.},
     TITLE = {Modular correspondence between dependent type theories and
              categories including pretopoi and topoi},
   JOURNAL = {Mathematical Structures in Computer Science},
    VOLUME = {15},
      YEAR = {2005},
    NUMBER = {6},
     PAGES = {1089--1149},
      ISSN = {0960-1295,1469-8072},
   MRCLASS = {03B15 (03G30 18B20)},
  MRNUMBER = {2189827},
MRREVIEWER = {Carsten\ Butz},
       DOI = {10.1017/S0960129505004962},
       URL = {https://doi.org/10.1017/S0960129505004962},
}

@incollection{DACCC,
     AUTHOR = {Lawvere, F. William},
     TITLE = {Diagonal arguments and cartesian closed categories},
 BOOKTITLE = {Category {T}heory, {H}omology {T}heory and their
              {A}pplications, {II} ({B}attelle {I}nstitute {C}onference,
              {S}eattle, {W}ash., 1968, {V}ol. {T}wo)},
    SERIES = {Lecture Notes in Math.},
    VOLUME = {No. 92},
     PAGES = {134--145},
 PUBLISHER = {Springer, Berlin-New York},
      YEAR = {1969},
   MRCLASS = {08.10 (02.00)},
  MRNUMBER = {242748},
MRREVIEWER = {H.\ Gonshor},
}

@incollection{EHCSAF,
      AUTHOR = {Lawvere, F. William},
     TITLE = {Equality in hyperdoctrines and comprehension schema as an
              adjoint functor},
 BOOKTITLE = {Applications of {C}ategorical {A}lgebra ({P}roc. {S}ympos.
              {P}ure {M}ath., {V}ol. {XVII}, {N}ew {Y}ork, 1968)},
    SERIES = {Proc. Sympos. Pure Math.},
    VOLUME = {XVII},
     PAGES = {1--14},
 PUBLISHER = {Amer. Math. Soc., Providence, RI},
      YEAR = {1970},
   MRCLASS = {18.10},
  MRNUMBER = {257175},
MRREVIEWER = {H.\ Gonshor},
}

@inproceedings{BirkedalCarboniRosolini,
  author={Birkedal, L. and Carboni, A. and Rosolini, G. and Scott, D.S.},
  booktitle={Proceedings. Thirteenth Annual IEEE Symposium on Logic in Computer Science (Cat. No.98CB36226)}, 
  title={Type theory via exact categories}, 
  year={1998},
  volume={},
  number={},
  pages={188-198},
  doi={10.1109/LICS.1998.705655}}

@book {CA,
    AUTHOR = {Freyd, Peter J. and Scedrov, Andre},
     TITLE = {Categories, allegories},
    SERIES = {North-Holland Mathematical Library},
    VOLUME = {39},
 PUBLISHER = {North-Holland Publishing Co., Amsterdam},
      YEAR = {1990},
     PAGES = {xviii+296},
      ISBN = {0-444-70368-3; 0-444-70367-5},
   MRCLASS = {18-02 (18Bxx)},
  MRNUMBER = {1071176},
MRREVIEWER = {Walter\ Tholen},
}

@article{FREY2023,
title = {Categories of partial equivalence relations as localizations},
journal = {Journal of Pure and Applied Algebra},
volume = {227},
number = {8},
pages = {107115},
year = {2023},
issn = {0022-4049},
doi = {https://doi.org/10.1016/j.jpaa.2022.107115},
url = {https://www.sciencedirect.com/science/article/pii/S0022404922001116},
author = {Jonas Frey}
}

@book{SAE,
      AUTHOR = {Johnstone, Peter T.},
     TITLE = {Sketches of an elephant: a topos theory compendium. {V}ol. 1},
    SERIES = {Oxford Logic Guides},
    VOLUME = {43},
 PUBLISHER = {The Clarendon Press, Oxford University Press, New York},
      YEAR = {2002},
      ISBN = {0-19-853425-6},
   MRCLASS = {18B25 (18-02)},
  MRNUMBER = {1953060},
MRREVIEWER = {Colin\ McLarty},
}

@book{SAE2,
      AUTHOR = {Johnstone, Peter T.},
     TITLE = {Sketches of an elephant: a topos theory compendium. {V}ol. 2},
    SERIES = {Oxford Logic Guides},
    VOLUME = {44},
 PUBLISHER = {The Clarendon Press, Oxford University Press, New York},
      YEAR = {2002},
      ISBN = {0-19-853425-6},
   MRCLASS = {18B25 (18-02)},
  MRNUMBER = {1953060},
MRREVIEWER = {Colin\ McLarty},
}

@incollection{TTT,
   AUTHOR = {Pitts, Andrew M.},
     TITLE = {Tripos theory in retrospect},
   BOOKTITLE = {Mathematical Structures in Computer Science. A Journal in the
              Applications of Categorical, Algebraic and Geometric Methods
              in Computer Science},
    VOLUME = {12},
      YEAR = {2002},
    NUMBER = {3},
     PAGES = {265--279},
      ISSN = {0960-1295,1469-8072},
   MRCLASS = {03G30 (18C10 68Q55)},
   
  MRNUMBER = {1909027},
       DOI = {10.1017/S096012950200364X},
       URL = {https://doi.org/10.1017/S096012950200364X},
}

@article{ECRT,
     AUTHOR = {Trotta, Davide},
     TITLE = {The existential completion},
  JOURNAL = {Theory and Applications of Categories},
    VOLUME = {35},
      YEAR = {2020},
     PAGES = {Paper No. 43, 1576--1607},
      ISSN = {1201-561X},
   MRCLASS = {18N15 (18C10 18C15 18N10)},
  MRNUMBER = {4162120},
MRREVIEWER = {Nicola\ Gambino},
}

@article{TT,
  AUTHOR = {Hyland, J. M. E. and Johnstone, P. T. and Pitts, A. M.},
     TITLE = {Tripos theory},
  JOURNAL = {Mathematical Proceedings of the Cambridge Philosophical
              Society},
    VOLUME = {88},
      YEAR = {1980},
    NUMBER = {2},
     PAGES = {205--231},
      ISSN = {0305-0041,1469-8064},
   MRCLASS = {03G30 (18B25)},
  MRNUMBER = {578267},
MRREVIEWER = {Anna\ Labella},
       DOI = {10.1017/S0305004100057534},
       URL = {https://doi.org/10.1017/S0305004100057534},
}

@inproceedings{Frey2014AFS,
  title={A fibrational study of realizability toposes},
  author={Frey, Jonas},
  journal={arXiv: Category Theory},
  year={2014},
  url={https://api.semanticscholar.org/CorpusID:119317371}
}

@article{EMMENEGGER,
title = {On the local cartesian closure of exact completions},
journal = {Journal of Pure and Applied Algebra},
volume = {224},
number = {11},
pages = {106414},
year = {2020},
issn = {0022-4049},
doi = {https://doi.org/10.1016/j.jpaa.2020.106414},
url = {https://www.sciencedirect.com/science/article/pii/S0022404920301134},
author = {Jacopo Emmenegger}
}

@book{van_Oosten_realizability,
  AUTHOR = {van Oosten, Jaap},
     TITLE = {Realizability: an introduction to its categorical side},
    SERIES = {Studies in Logic and the Foundations of Mathematics},
    VOLUME = {152},
 PUBLISHER = {Elsevier B. V., Amsterdam},
      YEAR = {2008},
     PAGES = {xvi+310},
      ISBN = {978-0-444-51584-1},
   MRCLASS = {03F60 (03-02 03G30 26E40)},
  MRNUMBER = {2479466},
MRREVIEWER = {Colin\ McLarty},
}

@article{EMMENEGGER2020,
AUTHOR = {Emmenegger, Jacopo and Pasquali, Fabio and Rosolini, Giuseppe},
     TITLE = {Elementary doctrines as coalgebras},
  JOURNAL = {Journal of Pure and Applied Algebra},
    VOLUME = {224},
      YEAR = {2020},
    NUMBER = {12},
     PAGES = {106445, 16},
      ISSN = {0022-4049,1873-1376},
   MRCLASS = {03G30 (03B10 03B20 03C45 18C20 18C50)},
  MRNUMBER = {4102179},
MRREVIEWER = {Anna\ Labella},
       DOI = {10.1016/j.jpaa.2020.106445},
       URL = {https://doi.org/10.1016/j.jpaa.2020.106445},
}

@article{BANASCHEWSKI199145,
AUTHOR = {Banaschewski, B. and Niefield, S. B.},
     TITLE = {Projective and supercoherent frames},
 BOOKTITLE = {Proceedings of the {C}onference on {L}ocales and {T}opological
              {G}roups ({C}ura\c cao, 1989)},
  JOURNAL = {Journal of Pure and Applied Algebra},
    VOLUME = {70},
      YEAR = {1991},
    NUMBER = {1-2},
     PAGES = {45--51},
      ISSN = {0022-4049,1873-1376},
   MRCLASS = {06D99 (06B35 08B30 54C55)},
  MRNUMBER = {1100504},
MRREVIEWER = {Peter\ Johnstone},
       DOI = {10.1016/0022-4049(91)90005-M},
       URL = {https://doi.org/10.1016/0022-4049(91)90005-M},
}

@article{MaiettiTrotta21,
 AUTHOR = {Maietti, Maria Emilia and Trotta, Davide},
     TITLE = {A characterization of generalized existential completions},
  JOURNAL = {Annals of Pure and Applied Logic},
    VOLUME = {174},
      YEAR = {2023},
    NUMBER = {4},
     PAGES = {Paper No. 103234, 37},
      ISSN = {0168-0072,1873-2461},
   MRCLASS = {18C10 (03G30 18C50 18D30)},
  MRNUMBER = {4530627},
MRREVIEWER = {Robert\ S.\ Lubarsky},
       DOI = {10.1016/j.apal.2022.103234},
       URL = {https://doi.org/10.1016/j.apal.2022.103234},
}

@misc{maschitrottao2025,
      title={A topos for extended {W}eihrauch degrees}, 
      author={Maschio, S. and Trotta, D.},
      year={2025},
      eprint={2505.08697},
      archivePrefix={arXiv},
      primaryClass={math.CT},
      url={https://arxiv.org/abs/2505.08697}, 
}

@article{maschiotrotta2024,
title = {On categorical structures arising from implicative algebras: From topology to assemblies},
journal = {Annals of Pure and Applied Logic},
volume = {175},
number = {3},
pages = {103390},
year = {2024},
issn = {0168-0072},
doi = {https://doi.org/10.1016/j.apal.2023.103390},
url = {https://www.sciencedirect.com/science/article/pii/S0168007223001471},
author = {Samuele Maschio and Davide Trotta}
}

@article{VANOOSTENmodifiedreal,
title = {The modified realizability topos},
journal = {Journal of Pure and Applied Algebra},
volume = {116},
number = {1},
pages = {273-289},
year = {1997},
issn = {0022-4049},
doi = {https://doi.org/10.1016/S0022-4049(97)00101-1},
url = {https://www.sciencedirect.com/science/article/pii/S0022404997001011},
author = {Jaap {van Oosten}}
}

@InProceedings{Hylandmodifiedreal,
author="Hyland, J. M. E. and Ong, C. -H. L.",
editor="Bezem, Marc
and Groote, Jan Friso",
title="Modified realizability toposes and strong normalization proofs",
booktitle="Typed Lambda Calculi and Applications",
year="1993",
publisher="Springer Berlin Heidelberg",
address="Berlin, Heidelberg",
pages="179--194",
isbn="978-3-540-47586-6"
}

@article{VANOOSTENereal,
title = {Extensional realizability},
journal = {Annals of Pure and Applied Logic},
volume = {84},
number = {3},
pages = {317-349},
year = {1997},
issn = {0168-0072},
doi = {https://doi.org/10.1016/S0168-0072(96)00050-4},
url = {https://www.sciencedirect.com/science/article/pii/S0168007296000504},
author = {Jaap {van Oosten}}
}

@incollection{HYLAND1982165,
title = {The Effective Topos},
editor = {A.S. Troelstra and D. {van Dalen}},
series = {Studies in Logic and the Foundations of Mathematics},
publisher = {Elsevier},
volume = {110},
pages = {165-216},
year = {1982},
booktitle = {The L. E. J. Brouwer Centenary Symposium},
author = {J.M.E. Hyland}
}

@article {MAIETTI_MASCHIO_2021,
    AUTHOR = {Maietti, Maria Emilia and Maschio, Samuele},
     TITLE = {A predicative variant of {H}yland's effective topos},
  JOURNAL = {The Journal of Symbolic Logic},
    VOLUME = {86},
      YEAR = {2021},
    NUMBER = {2},
     PAGES = {433--447},
      ISSN = {0022-4812,1943-5886},
   MRCLASS = {03G30 (03D20 03F30 03F55)},
  MRNUMBER = {4328013},
       DOI = {10.1017/jsl.2020.49},
       URL = {https://doi.org/10.1017/jsl.2020.49},
}

@article{MENNI2003287,
 AUTHOR = {Menni, Mat\'ias},
     TITLE = {A characterization of the left exact categories whose exact
              completions are toposes},
  JOURNAL = {Journal of Pure and Applied Algebra},
    VOLUME = {177},
      YEAR = {2003},
    NUMBER = {3},
     PAGES = {287--301},
      ISSN = {0022-4049,1873-1376},
   MRCLASS = {18A35 (18B25)},
  MRNUMBER = {1948025},
MRREVIEWER = {Ioan\ Tofan},
       DOI = {10.1016/S0022-4049(02)00261-X},
       URL = {https://doi.org/10.1016/S0022-4049(02)00261-X},
}

@article{MENNI2002187,
AUTHOR = {Menni, Mat\'ias},
     TITLE = {More exact completions that are toposes},
  JOURNAL = {Annals of Pure and Applied Logic},
    VOLUME = {116},
      YEAR = {2002},
    NUMBER = {1-3},
     PAGES = {187--203},
      ISSN = {0168-0072,1873-2461},
   MRCLASS = {18B25 (03D99 03G30 18A35)},
  MRNUMBER = {1900904},
MRREVIEWER = {Peter\ Johnstone},
       DOI = {10.1016/S0168-0072(01)00111-7},
       URL = {https://doi.org/10.1016/S0168-0072(01)00111-7},
}

@article{MENNI2007511,
AUTHOR = {Menni, Mat\'ias},
     TITLE = {Cocomplete toposes whose exact completions are toposes},
  JOURNAL = {Journal of Pure and Applied Algebra},
    VOLUME = {210},
      YEAR = {2007},
    NUMBER = {2},
     PAGES = {511--520},
      ISSN = {0022-4049,1873-1376},
   MRCLASS = {18B25 (18A35)},
  MRNUMBER = {2320014},
MRREVIEWER = {Dumitru\ Bu\c sneag},
       DOI = {10.1016/j.jpaa.2006.10.009},
       URL = {https://doi.org/10.1016/j.jpaa.2006.10.009},
}

@article{Hofstra2006,
AUTHOR = {Hofstra, Pieter J. W.},
     TITLE = {All realizability is relative},
  JOURNAL = {Mathematical Proceedings of the Cambridge Philosophical
              Society},
    VOLUME = {141},
      YEAR = {2006},
    NUMBER = {2},
     PAGES = {239--264},
      ISSN = {0305-0041,1469-8064},
   MRCLASS = {03G30 (03F55)},
  MRNUMBER = {2265872},
MRREVIEWER = {Carsten\ Butz},
       DOI = {10.1017/S0305004106009352},
       URL = {https://doi.org/10.1017/S0305004106009352},
}

@article{FREY2015232,
 AUTHOR = {Frey, Jonas},
     TITLE = {Triposes, q-toposes and toposes},
  JOURNAL = {Annals of Pure and Applied Logic},
    VOLUME = {166},
      YEAR = {2015},
    NUMBER = {2},
     PAGES = {232--259},
      ISSN = {0168-0072,1873-2461},
   MRCLASS = {03G30 (03B15 18D05)},
  MRNUMBER = {3281818},
MRREVIEWER = {Esfandiar\ Haghverdi},
       DOI = {10.1016/j.apal.2014.10.005},
       URL = {https://doi.org/10.1016/j.apal.2014.10.005},
}

@incollection{Fourman79,
AUTHOR = {Fourman, M. P. and Scott, D. S.},
     TITLE = {Sheaves and logic},
 BOOKTITLE = {Applications of sheaves ({P}roc. {R}es. {S}ympos. {A}ppl.
              {S}heaf {T}heory to {L}ogic, {A}lgebra and {A}nal., {U}niv.
              {D}urham, {D}urham, 1977)},
    SERIES = {Lecture Notes in Math.},
    VOLUME = {753},
     PAGES = {302--401},
 PUBLISHER = {Springer, Berlin},
      YEAR = {1979},
      ISBN = {3-540-09564-0},
   MRCLASS = {03C90 (03F55 03G30)},
  MRNUMBER = {555551},
MRREVIEWER = {B.\ W\polhk eglorz},
}

@article{higgs84, 
 AUTHOR = {Higgs, Denis},
     TITLE = {Injectivity in the topos of complete {H}eyting algebra valued
              sets},
 JOURNAL = {Canadian Journal of Mathematics},
    VOLUME = {36},
      YEAR = {1984},
    NUMBER = {3},
     PAGES = {550--568},
      ISSN = {0008-414X,1496-4279},
   MRCLASS = {18B25 (03G30 06D20)},
  MRNUMBER = {752984},
MRREVIEWER = {C.\ J.\ Mikkelsen},
       DOI = {10.4153/CJM-1984-034-4},
       URL = {https://doi.org/10.4153/CJM-1984-034-4},
}

@misc{frey2024,
      title={Uniform Preorders and Partial Combinatory Algebras}, 
      author={Frey, Jonas},
      year={2024},
      eprint={2403.17340},
      archivePrefix={arXiv},
      primaryClass={math.LO},
      url={https://arxiv.org/abs/2403.17340}, 
}

@article{colimitcomprobinson,
 AUTHOR = {Robinson, Edmund and Rosolini, Giuseppe},
     TITLE = {Colimit completions and the effective topos},
  JOURNAL = {The Journal of Symbolic Logic},
    VOLUME = {55},
      YEAR = {1990},
    NUMBER = {2},
     PAGES = {678--699},
      ISSN = {0022-4812,1943-5886},
   MRCLASS = {03G30 (03D65 18B25 18D30)},
  MRNUMBER = {1056382},
MRREVIEWER = {Colin\ McLarty},
       DOI = {10.2307/2274658},
       URL = {https://doi.org/10.2307/2274658},
}

@article{STREICHER_2013, 
title={Krivine’s classical realisability from a categorical perspective}, 
volume={23}, 
DOI={10.1017/S0960129512000989}, 
number={6}, 
journal={Mathematical Structures in Computer Science}, 
author={Streicher, Thomas}, 
year={2013}, 
pages={1234–1256}}

@phdthesis{Biering2008,
  author    = {Biering, Bodil},
  title     = {Dialectica Interpretations -- A Categorical Analysis},
  school    = {University of Copenhagen},
  year      = {2008},
  type      = {{PhD Thesis}}
}

@article{miquel_2020,
AUTHOR = {Miquel, Alexandre},
     TITLE = {Implicative algebras: a new foundation for realizability and
              forcing},
  JOURNAL = {Mathematical Structures in Computer Science},
    VOLUME = {30},
      YEAR = {2020},
    NUMBER = {5},
     PAGES = {458--510},
      ISSN = {0960-1295,1469-8072},
   MRCLASS = {03B40 (03E40 03F03 03G27 06A12 18B99)},
  MRNUMBER = {4122449},
MRREVIEWER = {Robert\ S.\ Lubarsky},
       DOI = {10.1017/s0960129520000079},
       URL = {https://doi.org/10.1017/s0960129520000079},
}

@inbook{MaiettiRosoliniRelating,
url = {https://doi.org/10.1515/9781501502620-014},
title = {Relating Quotient Completions via Categorical Logic},
booktitle = {Concepts of Proof in Mathematics, Philosophy, and Computer Science},
author = {Maria Emilia Maietti and Giuseppe Rosolini},
editor = {Dieter Probst and Peter Schuster},
publisher = {De Gruyter},
address = {Berlin, Boston},
pages = {229--250},
doi = {10.1515/9781501502620-014},
isbn = {9781501502620},
year = {2016},
lastchecked = {2025-11-02}
}

@book{moerdijk2000,
  title={Proper Maps of Toposes},
  author={Moerdijk, I. and Vermeulen, J.J.C.},
  isbn={9780821821688},
  series={American Mathematical Society: Memoirs of the American Mathematical Society},
  url={https://books.google.it/books?id=SwDUCQAAQBAJ},
  year={2000},
  publisher={American Mathematical Society}
}

@article{rogers2021,
    author ={Rogers, Morgan} ,
    title = {On Supercompactly and Compactly Generated Toposes},
    journal = {Theory and Applications of Categories},
    volume ={37},
    number ={32},
    pages={1017–1079},
    year = {2021}
}

@book{caramello2018,
  title={Theories, Sites, Toposes: Relating and Studying Mathematical Theories Through Topos-theoretic 'bridges'},
  author={Caramello, O.},
  isbn={9780191818752},
  url={https://books.google.it/books?id=7CuZAQAACAAJ},
  year={2018},
  publisher={Oxford University Press}
}

@article{hofstra2011,
  title={The {D}ialectica monad and its cousins},
  author={Hofstra, P.},
  journal={Models, logics, and higher-dimensional categories: A tribute to the work of {M}. {M}akkai},
  volume={53},
  pages={107--139},
  year={2011},
  url={https://web.archive.org/web/20210811043116/http://mysite.science.uottawa.ca/phofstra/dialectica.pdf}
}

@article{StreetFamily,
title = {The family approach to total cocompleteness and toposes},
author = {Ross Street},
year = {1984},
doi = {10.1090/S0002-9947-1984-0742429-3},
language = {English},
volume = {284},
pages = {355--369},
journal = {Transactions of the American Mathematical Society},
issn = {0002-9947},
publisher = {American Mathematical Society},
number = {1}
}

@phdthesis{WrigleyPHD,
   author = {Joshua Wrigley},
   school = {Università degli Studi dell'Insubria},
   title = {Doctrinal and Groupoidal Representations of
Classifying Topoi},
   year = {2024},
   url ={https://jlwrigley.github.io/thesis/thesis_wrigley.pdf}
    }

\end{document}